\documentclass[12pt]{amsart}

\usepackage{amsthm, amsmath, amsfonts, amssymb, graphicx, tikz-cd, calrsfs, hyperref, soul} 

\newtheorem{theorem}{Theorem}[section]
\newtheorem{lemma}[theorem]{Lemma}
\newtheorem{proposition}[theorem]{Proposition}
\newtheorem{corollary}[theorem]{Corollary}

\newtheorem{claim}[theorem]{Claim}
\theoremstyle{definition}
\newtheorem{definition}[theorem]{Definition}

\newtheorem{remark}[theorem]{Remark}

\newenvironment{proofclaim}{\paragraph{\emph{Proof of the Claim}.}}{\hfill$\qed$\\}

\newcommand{\func}[1]{\operatorname{#1}}
\def\bal{\boldsymbol{\mathit{ba}\ell}}
\def\ubal{\boldsymbol{\mathit{uba}\ell}}
\def\dbal{\boldsymbol{\mathit{dba}\ell}}

\def\BA{\sf{BA}}
\def\SL{\sf{SLat}_0^1}
\def\Frm{\sf{Frm}}
\def\KHaus{\sf{KHaus}}

\def\BE{\mathcal{B}}
\def\L{\mathcal{L}}

\def\B{\mathfrak{B}}
\def\Id{\func{Id}}
\def\Int{{\sf{int}}}
\def\Cl{{\sf{cl}}}
\def\filt{{\sf{Filt}}}
\def\Up{{\sf{Up}}}

\def\arch{{\sf{Arch}}}
\def\D{{\sf{Dn}}}
\def\RO{{\sf{RO}}}
\def\Clop{{\sf{Clop}}}

\newcommand{\el}[1]{[{#1}]}
\newcommand{\ar}[1]{\langle {#1}\rangle}
\def\up{{\uparrow}}
\def\down{{\downarrow}}

\makeatletter
\providecommand{\bigsqcap}{%
  \mathop{%
    \mathpalette\@updown\bigsqcup
  }%
}
\newcommand*{\@updown}[2]{%
  \rotatebox[origin=c]{180}{$\m@th#1#2$}%
}
\makeatother

\title[A point-free approach to canonical extensions]{A point-free approach to canonical extensions of boolean algebras and bounded archimedean $\ell$-algebras}

\author{G. Bezhanishvili}
\address{New Mexico State University}
\email{guram@nmsu.edu}

\author{L. Carai}
\address{New Mexico State University}
\email{lcarai@nmsu.edu}

\author{P. Morandi}
\address{New Mexico State University}
\email{pmorandi@nmsu.edu}

\date{}


\begin{document}

\begin{abstract}
In \cite{BH20} an elegant choice-free construction of a canonical extension of a boolean algebra $B$ was given as the boolean algebra of regular open subsets of the Alexandroff topology on the poset of proper filters of $B$. We make this construction point-free by replacing the Alexandroff space of proper filters of $B$ with the free frame $\L$ generated by the bounded meet-semilattice of all filters of $B$ (ordered by reverse inclusion) and prove that the booleanization of $\L$ is a canonical extension of $B$. Our main result generalizes this approach to the category $\bal$ of bounded archimedean $\ell$-algebras, thus yielding a point-free construction of canonical extensions in $\bal$. We conclude by showing that the algebra of normal functions on the Alexandroff space of proper archimedean $\ell$-ideals of $A$ is a canonical extension of $A\in\bal$, thus providing a generalization of the result of \cite{BH20} to $\bal$.
\end{abstract}

\maketitle

\tableofcontents

\section{Introduction} \label{sec: intro}

The theory of canonical extensions originates from the pioneering work of J\'onsson and Tarski \cite{JT51}. Originally it was defined for boolean algebras with operators, but was later generalized to distributive lattices with operators \cite{GJ94,GJ04}, lattices with operators \cite{GH01}, and even to posets with operators \cite{GP08,GJP13}. 

One of the most convenient (albeit neither choice-free nor point-free) ways to describe a canonical extension of a boolean algebra $B$ is using Stone duality. If $X$ is the Stone space of $B$, then $B$ is isomorphic to the boolean algebra $\Clop(X)$ of clopen subsets of $X$, and the pair $(\wp(X),e)$ is a canonical extension of $B$ where $\wp(X)$ is the powerset of $X$ and $e:\Clop(X)\to\wp(X)$ is the identity embedding.

This approach generalizes naturally from Stone spaces to compact Hausdorff spaces. Let $X$ be compact Hausdorff and $C(X)$ the ring of continuous (necessarily bounded) real-valued functions on $X$. In \cite{BMO18c} a canonical extension of $C(X)$ was described as the pair $(B(X),e)$ where $B(X)$ is the ring of all bounded real-valued functions on $X$ and $e:C(X)\to B(X)$ is the identity embedding. More generally, if $A$ is a bounded archimedean $\ell$-algebra, by Gelfand duality $A$ embeds into $C(X)$, where $X$ is the compact Hausdorff space of maximal $\ell$-ideals of $X$, and the pair $(B(X),\zeta)$ is a canonical extension of $A$, where $\zeta:A\to C(X)\subseteq B(X)$ is the embedding of $A$ into $C(X)$ (see Sections~\ref{sec: bal} and~\ref{sec: canonical extensions in bal point-free} for details). 

This approach to canonical extensions is neither choice-free nor point-free. An elegant choice-free approach to canonical extensions of boolean algebras was developed in \cite{BH20} where a canonical extension of a boolean algebra $B$ was constructed as the boolean algebra of regular open sets of the Alexandroff space of proper filters of $B$ (ordered by inclusion). Our first aim is to make the construction of \cite{BH20} point-free. For this we utilize the well-known fact that the free frame on a meet-semilattice $M$ with top is isomorphic to the downsets of $M$ (see, e.g., \cite[Prop.~IV.2.3]{PP12}).
For a boolean algebra $B$, let $\filt(B)$ be the co-frame of filters ordered by reverse inclusion. We view $\filt(B)$ as a bounded meet-semilattice, and show that the free frame $\L$ on $\filt(B)$ is isomorphic to the Alexandroff space of proper filters of $B$ ordered by inclusion (see Corollary~\ref{cor: Nick-Wes frame}). From this we derive that the booleanization $\B(\L)$ of $\L$ is a canonical extension of $B$ (see Theorem~\ref{thm:canonical BA}).

Our second aim is to generalize this point-free approach to the category $\bal$ of bounded archimedean $\ell$-algebras.  The interest in this category stems from the fact that $\bal$ provides an algebraic counterpart of the category $\KHaus$ of compact Hausdorff spaces. Indeed, by Gelfand duality, there is a dual adjunction between $\KHaus$ and $\bal$, which restricts to a dual equivalence between $\KHaus$ and the full subcategory $\ubal$ of $\bal$ consisting of uniformly complete algebras in $\bal$ (see, e.g., \cite{BMO13a}). Generalizing our point-free approach 
to $\bal$ requires additional machinery. Let $A\in\bal$. Following the work of Banaschewski \cite{Ban97,Ban05a}, we work with archimedean $\ell$-ideals of $A$ (see Section~\ref{sec: arch}). 
Let $\arch(A)$ be the frame of archimedean $\ell$-ideals ordered by inclusion. Assuming the Axiom of Choice (AC), $\arch(A)$ is isomorphic to the frame of opens of the space of maximal $\ell$-ideals of $A$ (see Remark~\ref{ref: arch iso to opens}). 

Viewing $\arch(A)$ as a bounded meet-semilattice, let $\L$ be the free frame generated by $\arch(A)$, and let $\BE(\L)$ be the free boolean extension of $\L$ (see, e.g., \cite[Sec.~V.4]{BD74}). We employ the concepts of Specker algebra and Dedekind completion (see Section~\ref{sec: bal} for details), which play an important role in the study of $\bal$. We associate with $\BE(\L)$ the Specker algebra $\mathbb{R}[\BE(\L)]$ and prove that the Dedekind completion $D(\mathbb{R}[\BE(\L)])$ of $\mathbb{R}[\BE(\L)]$ is a canonical extension of $A$ (see Theorem~\ref{thm: main}).  This is our main result and yields a point-free construction of canonical extensions in $\bal$. Its proof requires a number of technical calculations about archimedean $\ell$-ideals. In order to not break the flow, we move these calculations to an appendix.

Finally, we show that the algebra of normal real-valued functions on the Alexandroff space of proper archimedean $\ell$-ideals ordered by inclusion is a canonical extension of $A\in\bal$ (see Theorem~\ref{thm:normal}). On the one hand, this provides a generalization of the construction of \cite{BH20}. On the other hand, assuming (AC), this algebra of normal functions is isomorphic to the algebra of bounded real-valued functions on the set of maximal $\ell$-ideals of $A$, thus yielding the result of \cite{BMO18c}. 

\section{Canonical extensions of boolean algebras point-free} \label{sec: canonical extensions in BA}

In this section we show how to give a point-free description of canonical extensions of boolean algebras. Let $\BA$ be the category of boolean algebras and boolean homomorphisms. The next definition is well known (see, e.g., \cite[Sec.~2]{GH01}).

\begin{definition} 
Let $B$ be a boolean algebra, $C$ a complete boolean algebra, and $e : B \to C$ a $\BA$-monomorphism.
\begin{enumerate}
\item We call $e$ \emph{compact} if whenever $S, T \subseteq B$ with $\bigwedge e[S]  \le \bigvee e[T]$, there are finite $S_0 \subseteq S$ and $T_0 \subseteq T$ with $\bigwedge S_0 \le \bigvee T_0$.
\item We call $e$ \emph{dense} if each element of $C$ is a join of meets from $e[B]$.
\item We say that the pair $(C,e)$ is a \emph{canonical extension} of $B$ if $e$ is dense and compact.
\end{enumerate}
\end{definition}

\begin{remark} \label{rem: compactness for BA}
It is straightforward to see that the compactness condition is equivalent to each of the following two conditions.
\begin{enumerate}
\item If $T \subseteq B$ with $\bigvee e[T] = 1$, then there is a finite $T_0 \subseteq T$ with $\bigvee T_0 = 1$.
\item If $S \subseteq B$ with $\bigwedge e[S] = 0$, then there is a finite $S_0 \subseteq S$ with $\bigwedge S_0 = 0$.
\end{enumerate}
In fact, (1) is the original definition of compactness in \cite{JT51}. We will use (2) in the proof of Theorem~\ref{thm:canonical BA}.
\end{remark}

J\'onsson and Tarski \cite{JT51} utilized Stone duality to show that each boolean algebra has a canonical extension, which is unique up to isomorphism. This requires the use of (AC).
An elegant choice-free description of canonical extensions of boolean algebras was given in \cite{BH20}.  Let $B \in \BA$ and let $X$ be the set of proper filters of $B$, ordered by inclusion. View $X$ as an Alexandroff space where opens are the upsets of $X$ (so $U$ is open provided $x\in U$ and $x\le y$ imply $y\in U$). Let $\RO(X)$ be the boolean algebra of regular open subsets of $X$. Define $e : B \to \RO(X)$ by $e(b) = \{ x \in X \mid b \in x\}$. Then $(\RO(X),e)$ is a canonical extension of $B$ \cite[Thm.~8.27]{BH20}. 

We give a point-free description of the construction in \cite{BH20}. To do so we recall some basic notions from point-free topology. We refer the reader to \cite{PP12} for the details. 

A \emph{frame} or \emph{locale} is a complete distributive lattice $L$ satisfying the infinite distributive law $a\wedge\bigvee S = \bigvee\{ a\wedge s \mid s\in S\}$ for each $a\in L$ and $S\subseteq L$. A \emph{frame homomorphism} is a map $f:L\to M$ between two frames that preserves finite meets and arbitrary joins. For a frame $L$ and $a \in L$, let $a^*:=\bigvee\{s\in L \mid a\wedge s=0\}$ be the \emph{pseudocomplement} of $a$. The set $\B(L) := \{ a^{**} \mid a\in L \}$ 
is called the \emph{booleanization} of $L$. It is a complete boolean algebra, where infinite joins and meets are calculated by
\[
\bigsqcup S = \left(\bigvee S\right)^{**} \mbox{ and } \bigsqcap S = \bigwedge S.
\]
The \emph{well inside relation} $\prec$ is defined on a frame $L$ by $a \prec b$ if $a^* \vee b = 1$. Then $L$ is \emph{regular} if $a = \bigvee \{ s \in L \mid s \prec a\}$ for each $a \in L$. Also, $L$ is \emph{compact} if whenever $S \subseteq L$ with $\bigvee S = 1$, there is a finite $S_0 \subseteq S$ with $\bigvee S_0 = 1$. 

\begin{definition}
Let $B \in \BA$. We denote by $\filt(B)$ the set of filters of $B$, ordered by reverse inclusion.
\end{definition}

\begin{remark}
As we will see below, we are using reverse inclusion on $\filt(B)$ in order for the map $B \to \filt(B)$ which sends $b$ to $\up b$ to be order preserving. 
\end{remark}

Let $\sf{SLat}^1$ be the category of meet-semilattices with top with meet-semilattice morphisms preserving the top. The free frame on $M \in \sf{SLat}^1$ is isomorphic to the frame of downsets $\D(M)$ of $M$ (see, e.g., \cite[Prop.~IV.2.3]{PP12}), where we recall that $D$ is a downset if whenever $x \le y$ and $y \in D$, we have $x \in D$.  Let $\SL$ be the category of meet-semilattices with top and bottom, with meet-semilattice morphisms preserving both the top and bottom. 

\begin{proposition} \label{prop: free frame on M}
Let $M \in \SL$. Then $\D(M \setminus \{0\})$ is isomorphic to the free frame on $M$.
\end{proposition}

\begin{proof}
Define $i : M \to \D(M \setminus \{0\})$ by $i(m) = \down m \setminus \{0\}$. It is straightforward to see that $i$ is a $\SL$-morphism. Let $L$ be a frame and $f : M \to L$ a $\SL$-morphism. Define $\varphi : \D(M \setminus \{0\}) \to L$ by $\varphi(D) = \bigvee \{ f(m) \mid m \in D\}$. We show that $\varphi$ is a frame homomorphism satisfying $\varphi \circ i = f$, and that $\varphi$ is uniquely determined by these properties.
\[
\begin{tikzcd}
M \arrow[r, "i"] \arrow[dr, "f"'] & D(M \setminus \{0\}) \arrow[d, "\varphi"] \\
& L
\end{tikzcd}
\]

First, $\varphi(\varnothing) = \bigvee \varnothing = 0$ and $\varphi(M \setminus \{0\}) = f(1) = 1$. Next, let $D_1, D_2$ be downsets of $M \setminus \{0\}$. We have
\begin{align*}
\varphi(D_1) \wedge \varphi(D_2) &= \bigvee \{ f(m) \mid m \in D_1 \} \wedge \bigvee \{ f(n) \mid n \in D_2 \} \\
&= \bigvee \{ f(m) \wedge f(n) \mid m \in D_1, n \in D_2\} \\
&= \bigvee \{ f(m \wedge n) \mid m \in D_1, n \in D_2\} \\
&= \bigvee \{ f(p) \mid p \in D_1 \cap D_2 \} = \varphi(D_1 \cap D_2).
\end{align*}
Also, let $\{D_\gamma \mid \gamma \in \Gamma\}$ be a family of downsets. Then
\begin{align*}
\varphi\left(\bigcup\left\{ D_\gamma \mid \gamma \in \Gamma \right\}\right) &= \bigvee \left\{ f(m) \mid m \in \bigcup D_\gamma \right\} \\
&= \bigvee \left\{ \bigvee \{ f(m) \mid m \in D_\gamma \} \mid \gamma \in \Gamma \right\} \\
&= \bigvee \{\varphi(D_\gamma) \mid \gamma \in \Gamma \}.
\end{align*}
Therefore, $\varphi$ is a frame homomorphism. It is clear from the definition that $\varphi(i(m)) = f(m)$ for each $m\in M$, so $\varphi \circ i = f$. Finally, since $D = \bigcup \{ i(m) \mid m \in D\}$, it follows that $\varphi$ is uniquely determined by the equation $\varphi \circ i = f$. Thus, $\D(M \setminus \{0\})$ is, up to isomorphism, the free frame on $M \in \SL$.
\end{proof}

\begin{remark} \label{rem: free frame}
Let $M \in \SL$ and let $(\L, i)$ be the free frame on $M$. Since $i$ preserves finite meets, every element of $\L$ is a join of elements from $i[M]$. 
\end{remark}

\begin{corollary} \label{cor: Nick-Wes frame}
Let $B \in \BA$. The frame of upsets of proper filters of $B$, ordered by reverse inclusion, is isomorphic the free frame on $\filt(B) \in \SL$.
\end{corollary}

\begin{proof}
By Proposition~\ref{prop: free frame on M}, the free frame on $\filt(B)$ is isomorphic to the frame of all downsets of $ \filt(B) \setminus \{B\}$, which is isomorphic to the frame of upsets of $X$.
\end{proof}

Let $B \in \BA$ and let $\L$ be the free frame on the bounded meet-semilattice $\filt(B)$ with the associated map $i : \filt(B) \to \L$. Define $e : B \to \L$ by $e(b) = i(\up b)$.  

\begin{lemma} \label{lem: i goes to B(L)}
If $b \in B$, then $i(\up b)^* = i(\up \lnot b)$. Consequently, $i(\up b) \in \B(\L)$.
\end{lemma}

\begin{proof}
Since $i$ is a $\SL$-morphism and the order on $\filt(B)$ is reverse inclusion, 
\[
i(\up b) \wedge i(\up \lnot b) = i(\up b \vee \up \lnot b) = i(\up (b \wedge \lnot b))  = i(B).
\] 
Because $B$ is the bottom of $\filt(B)$, we obtain that $i(\up b) \wedge i(\up \lnot b) = 0$. Let $x \in \L$ with $i(\up b) \wedge x = 0$. To show $x \le i(\up \lnot b)$, by Remark~\ref{rem: free frame},  $x$ is a join from $i[\filt(B)]$. Since the order on $\filt(B)$ is reverse inclusion, it suffices to show that if $F$ is a filter of $B$, then $i(\up b) \wedge i(F) = 0$ implies $\up \lnot b \subseteq F$. If $i(\up b) \wedge i(F) = 0$, then $i(\up b \vee F) = 0$, so $\up b \vee F = B$ as $i$ is one-to-one. Therefore, there is $a \in F$ with $a \wedge b = 0$. Thus, $a \le \lnot b$, and hence $\up \lnot b \subseteq F$, as desired. This shows that $i(\up b)^* = i(\up \lnot b)$. From this we see that $i(\up b)^{**} = i(\up \lnot b)^* = i (\up \lnot\lnot b) = i(\up b)$, so $i(\up b) \in \B(\L)$.
\end{proof}

\begin{theorem} \label{thm:canonical BA}
For $B\in\BA$, the pair $(\B(\L),e)$ is a canonical extension of $B$.
\end{theorem}

\begin{proof}
By Lemma~\ref{lem: i goes to B(L)}, $e(b) \in \B(\L)$, so $e : B \to \B(\L)$ is well defined. If $b \ne c$, then $\up b \ne \up c$, so $i(\up b) \ne i(\up c)$. Therefore, $e$ is one-to-one. To see that $e$ is a $\BA$-morphism, let $b, c \in B$. Then
\[
e(b \wedge c) = i(\up (b \wedge c)) = i(\up b \vee \up c) 
= i(\up b) \wedge i(\up c) = e(b) \wedge e(c),
\]
so $e$ preserves meet. Also, by Lemma~\ref{lem: i goes to B(L)}, $e(\lnot b) = i(\up \lnot b) = i(\up b)^* = e(b)^*$. Thus, $e$ preserves negation, and hence is a $\BA$-morphism.

We next show that $e$ is dense. Since every element of $\B(\L)$ is a join from $i[\filt(B)]$, it is enough to show that if $F$ is a filter of $B$, then $i(F)$ is a meet from $e[B]$. We show that $i(F) = \bigwedge \{ e(b) \mid b \in F\}$.
First, if $b \in F$, then $\up b \subseteq F$, so 
$i(F) \le i(\up b) = e(b)$. Therefore, $i(F) \le \bigwedge \{ e(b) \mid b \in F\}$. For the reverse inequality, suppose that $G$ is a filter with $i(G) \le \bigwedge \{ e(b) \mid b \in F\}$. Then $i(G) \le i(\up b)$, so $\up b \subseteq G$, and hence $b \in G$ for each $b \in F$. This implies that $F \subseteq G$, and so $i(G)\le i(F)$. Therefore, if $x \in \L$ with $x \le \bigwedge \{ e(b) \mid b \in F\}$, then $x \le i(F)$, showing that $i(F) = \bigwedge \{ e(b) \mid b \in F\}$. Thus, $i(F)$ is a meet from $e[B]$, and so each element of $\L$ is a join of meets from $e[B]$. Consequently, $e$ is dense.

Finally, we show that $e$ is compact. 
Let $S \subseteq B$ with $\bigwedge e[S] = 0$. Then
\[
0 = \bigwedge e[S]  = \bigwedge \{ i(\up s)  \mid s \in S \}.
\]
Let $F$ be the filter generated by $S$. Then $\up s \subseteq F$, and hence $i(F) \le i(\up s)$ for each $s \in S$. This forces $i(F) = 0$, so 
$F = B$. Therefore, there are $s_1, \dots, s_n \in S$ with $s_1 \wedge \cdots \wedge s_n = 0$. Thus, $e$ is compact by Remark~\ref{rem: compactness for BA}, and hence $(\B(\L),e)$ is a canonical extension of $B$.
\end{proof}

We can now derive the result of \cite{BH20}. By Corollary~\ref{cor: Nick-Wes frame}, the Alexandroff topology $\Up(X)$ on $X$ is isomorphic to $\L$. Consequently, $\RO(X) \cong \B(\L)$. Moreover, define $f : \filt(B) \to \Up(X)$ by $f(F) 
= \{ G \in X \mid F \subseteq G\}$. It is easy to see that $f$ is a $\SL$-morphism, so induces a frame homomorphism $\varphi : \L \to \Up(X)$ satisfying $\varphi \circ i = f$.
\[
\begin{tikzcd}
\filt(B) \arrow[r, "i"] \arrow[dr, "f"'] & \L \arrow[d, "\varphi"] \\
& \Up(X)
\end{tikzcd}
\]

If $b \in B$, then
\[
\varphi(e(b)) = \varphi(i(\up b)) = f(\up b) = \{ G \in X \mid \up b \subseteq G\} = \{ G \in X \mid b \in G\},
\]
which is the map defined in \cite{BH20}. This yields an explicit isomorphism between our construction and that in \cite{BH20}.

\section{Bounded archimedean $\ell$-algebras} \label{sec: bal}

In this section we recall several basic facts about bounded archimedean $\ell$-algebras. We assume the reader's familiarity with $\ell$-rings (lattice-ordered rings) and $\ell$-algebras (lattice-ordered algebras). We use \cite[Ch.~XIII and onwards]{Bir79} as our main reference for $\ell$-rings and \cite{HJ61,BMO13a} as our main references for $\ell$-algebras. 
All rings are assumed to be commutative and unital (have multiplicative identity $1$).

\begin{definition}
\label{def: bal}
\begin{enumerate}
\item[]
\item An $\ell$-algebra $A$ is \emph{bounded} if for each $a \in A$ there is an integer $n \ge 1$ such that $a \le n\cdot 1$ (that is, $1$ is 
a \emph{strong order unit}). 
\item An $\ell$-algebra $A$ is \emph{archimedean} if for each $a, b \in A$, whenever $n\cdot a \le b$ for each $n \ge 1$, then $a \le 0$.
\item A \emph{$\bal$-algebra} is a bounded archimedean $\ell$-algebra and a $\bal$-morphism is a unital $\ell$-algebra homomorphism. Let $\bal$ be the category of $\bal$-algebras and $\bal$-morphisms.
\end{enumerate}
\end{definition}

Let $A\in\bal$. If $A\ne 0$, then we view $\mathbb{R}$ as an $\ell$-subalgebra of $A$ by identifying $r\in\mathbb R$ with $r\cdot 1\in A$.

\begin{definition}
Let $A\in\bal$ and $a\in A$.
\begin{enumerate}
\item Define the \emph{positive} and \emph{negative} parts of $a$ by
\[
a^+ = a \vee 0 \textrm{ and }  a^- = (-a) \vee 0 = -(a \wedge 0).
\]
\item Define
the \emph{absolute value} of $a$ by
\[
|a|=a\vee(-a).
\]
\item Define the \emph{norm} of $a \in A$ by
\[
||a||=\inf\{r\in\mathbb R : |a|\le r\}.
\]
\end{enumerate}
We call $A$ \emph{uniformly complete} if the norm is complete. Let $\ubal$ be the full subcategory of $\bal$ consisting of uniformly complete objects.
\end{definition}

\begin{theorem} [Gelfand duality \cite{GN43,Sto40}] \label{thm: GD}
There is a dual adjunction between $\bal$ and $\KHaus$ which restricts to a dual equivalence between $\KHaus$ and $\ubal$.
\[
\begin{tikzcd}
\ubal \arrow[rr, hookrightarrow] && \bal \arrow[dl, "(-)_*"]  \arrow[ll, bend right = 20] \\
&  \KHaus \arrow[ul,  "(-)^*"] &
\end{tikzcd}
\]
\end{theorem}

The contravariant functors $(-)^*:\KHaus \to \bal$ and $(-)_*:\bal\to\KHaus$ establishing the dual adjunction of Theorem~\ref{thm: GD} are defined as follows. For a compact Hausdorff space $X$ let $X^*$ be the ring $C(X)$ of (necessarily bounded) continuous real-valued functions on $X$. For a continuous map $\varphi:X\to Y$ let $\varphi^*:C(Y)\to C(X)$ be defined by $\varphi^*(f)=f\circ\varphi$ for each $f\in C(Y)$. Then $(-)^*:\KHaus\to\bal$ is a well-defined contravariant functor.

For $A\in\bal$, we recall that an ideal $I$ of $A$ is an \emph{$\ell$-ideal} if $|a|\le|b|$ and $b\in I$ imply $a\in I$, and that $\ell$-ideals are exactly the kernels of $\ell$-algebra homomorphisms. Let $Y_A$ be the space of maximal $\ell$-ideals of $A$, whose closed sets are exactly sets of the form
\[
Z_\ell(I) = \{M\in Y_A\mid I\subseteq M\},
\]
where $I$ is an $\ell$-ideal of $A$. The space $Y_A$ is often referred to as the \emph{Yosida space} of $A$, and it is well known that $Y_A\in\KHaus$. Let $A_*=Y_A$ and for a $\bal$-morphism $\alpha$ let $\alpha_*=\alpha^{-1}$. Then $(-)_*:\bal\to\KHaus$ is a well-defined contravariant functor, and the functors $(-)_*$ and $(-)^*$ yield a dual adjunction between $\bal$ and $\KHaus$.

Moreover, for $X\in\KHaus$ we have that $\varepsilon_X: X\to Y_{C(X)}$
is a homeomorphism where
\[
\varepsilon_X(x)=\{f\in C(X) \mid f(x)=0\}.
\]
Furthermore, for $A\in\bal$ define $\zeta_A :A\to C(Y_A)$ by $\zeta_A(a)(M)=r$ where $r$ is the unique real number satisfying $a+M=r+M$. Then $\zeta_A$ is a monomorphism in $\bal$ separating points of $Y_A$. Thus, by the Stone-Weierstrass theorem, we have:

\begin{theorem}\label{prop: SW}
\begin{enumerate}
\item[]
\item The uniform completion of $A \in \bal$ is $\zeta_A : A \to C(Y_A)$. Therefore, if $A$ is uniformly complete, then $\zeta_A$ is an isomorphism.
\item $\ubal$ is a reflective subcategory of $\bal$, and the reflector $\zeta : \bal \to \ubal$ assigns to each $A \in \bal$ its uniform completion $C(Y_A) \in \ubal$.
\end{enumerate}
\end{theorem}

Consequently, the dual adjunction restricts to a dual equivalence between $\ubal$ and $\KHaus$, yielding Gelfand duality.

For the results in Section~\ref{sec: canonical extensions in bal point-free} we recall the definition of Dedekind algebras, Dedekind completions, and Specker algebras (see, e.g., \cite{BMO13a, BMO13b}). 

\begin{definition} \label{dc def}
Let $A \in\bal$.
\begin{enumerate}
\item $A$ is \emph{Dedekind complete} if every subset of $A$ bounded above has a least upper bound (and hence every subset of $A$ bounded below has a greatest lower bound).
\item $A$ is a \emph{Dedekind $\bal$-algebra} if $A$ is Dedekind complete.
\item Let $\dbal$ be the full subcategory of $\bal$ consisting of Dedekind $\bal$-algebras.
\end{enumerate}
\end{definition}

If $A \in \bal$, then there is a unique up to isomorphism Dedekind $\bal$-algebra $D(A) \in \bal$ such that $A$ embeds into $D(A)$ and the image is join dense in $D(A)$. This result was proved by Nakano \cite[Sec.~31]{Nak50} in the setting of vector lattices and by Johnson \cite[p.~493]{Joh65} in the setting of $f$-rings. It was adapted to $\bal$ in \cite{BMO13b}.

\begin{definition}
For $A\in\bal$, we call $D(A)$ the \emph{Dedekind completion} of $A$. Throughout this paper we will identify $A$ with its image in $D(A)$.
\end{definition}

The following theorem provides a characterization of Dedekind completions in $\bal$ (see, e.g., \cite[Thm.~3.1]{BMO13b}).

\begin{theorem} \label{thm: Ded} 
Let $A \in \bal$ and $D \in \dbal$. If $\alpha : A \to D$ is a $\bal$-monomorphism such that every element of $D$ is a join from $\alpha[A]$, then there is a $\bal$-isomorphism $\varphi : D(A) \to D$ with $\varphi|_A = \alpha$.
\end{theorem}

Let $A \in \bal$. Since $A$ is a commutative ring, it is well known that the set $\Id(A)$ of idempotents of $A$ is a boolean algebra under the operations
\[
e \vee f = e + f - ef, \quad e \wedge f = ef, \quad \lnot e = 1-e.
\]
The ordering on the boolean algebra $\Id(A)$ is the restriction of the ordering on $A$, and hence $e \in \Id(A)$ implies that $0 \le e \le 1$.

\begin{definition}\label{def of Specker} \cite[Sec.~5]{BMO13a}
We call $A\in\bal$ a \emph{Specker $\bal$-algebra} if $A$ is generated by $\Id(A)$.
\end{definition}

\begin{remark}
For the history of the notion of a Specker algebra see, e.g., \cite{BMO20e}. 
\end{remark}

Specker algebras can be characterized by the following construction, the origins of which go back to Bergman \cite{Ber72} and Rota \cite{Rot73}.

\begin{definition}\label{def: R[B]} \cite[Def.~2.4]{BMMO15a}
Let $B$ be a Boolean algebra. We denote by $\mathbb{R}[B]$ the quotient ring $\mathbb{R}[\{x_e \mid e\in B\}]/I_B$ of the polynomial ring over $\mathbb{R}$ in variables indexed by the elements of $B$ modulo the ideal $I_B$ generated by the following elements, as $e, f$ range over $B$: $$x_{e\wedge f} - x_e x_f, \ \ x_{e\vee f} - (x_e + x_f - x_e x_f), \ \ x_{\lnot e} - (1-x_e), \ \ x_0.$$
\end{definition}

For $e\in B$ we abuse notation and identify $x_e$ with its image in $\mathbb{R}[B]$. Considering the generators of $I_B$, for all $e,f\in B$, we have that
\[
x_{e\wedge f}=x_e x_f, \ \ x_{e\vee f}= x_e + x_f - x_e x_f, \ \ x_{\lnot e}=1-x_e, \ \ x_0=0.
\]

\begin{theorem} \label{thm: UMP}
Let $A\in\bal$. 
\begin{enumerate} 
\item \label{UMP: 1} \cite[Lem.~3.2(4)]{BMMO15a} The map sending $e \in B$ to $x_e \in \mathbb{R}[B]$ is a $\BA$-isomorphism from $B$ to $\Id(\mathbb{R}[B])$.
\item \label{UMP: 2}\cite[Lem.~2.5]{BMMO15a} For $B \in \BA$ and a $\BA$-morphism $\tau : B \to \Id(A)$, there is a unique $\bal$-morphism $\sigma : \mathbb{R}[B] \to A$ such that $\sigma(x_b) = \tau(b)$ for each $b \in B$.
\item \label{UMP: 3}\cite[Thm.~2.7]{BMMO15a}
$A$ is a Specker $\bal$-algebra iff $A \cong \mathbb{R}[B]$ for some $B \in \BA$.
\end{enumerate}
\end{theorem}

\begin{remark}
Associating $\Id(A)$ with each $A \in \bal$ extends to a covariant functor $\Id : \bal \to \BA$. Definition~\ref{def: R[B]} gives rise to a covariant functor $\func{Sp} : \BA \to \bal$ which is left adjoint to $\Id$. The functors $\Id$ and $\func{Sp}$ yield an equivalence between $\BA$ and the full subcategory of $\bal$ consisting of Specker $\bal$-algebras (see \cite[Sec.~3]{BMMO15a} for details).
\end{remark}

We will use the following two facts about $\Id(A)$. The proof of the first one can be found in \cite[Lem.~4.9(6)]{BMMO15b}, and we give a short proof of the second.

\begin{remark} \label{rem: idempotent facts}
Let $A \in \bal$.
\begin{enumerate}
\item \label{idempotent facts: 1}Let $0 \ne e,f \in \Id(A)$ and $0 < r,s \in \mathbb{R}$. If $re \le sf$, then $r \le s$ and $e \le f$.
\item \label{idempotent facts: 2}If $e \in \Id(A)$ and $0 \le a \in A$ with $1 = e \vee a$, then $\lnot e \le a$. Indeed, since $\lnot e$ is an idempotent, $0 \le \lnot e \le 1$, so 
\[
\lnot e = \lnot e \wedge (e \vee a) = (\lnot e \wedge e) \vee (\lnot e \wedge a) = 0 \vee (\lnot e \wedge a) = \lnot e \wedge a
\]
because $0 \le a$.
Thus, $\lnot e \le a$. 
\end{enumerate}
\end{remark}

In the next remark we collect together some elementary facts about $\mathbb{R}[B]$. We will use them frequently in our proofs.

\begin{remark} \label{rem: Specker facts}
Let $e, f \in B$.
\begin{enumerate}
\item \label{Specker facts: 1} $x_e \vee x_f = x_{e \vee f}$ and $x_e \wedge x_f = x_{e \wedge f}$. 
\item \label{Specker facts: 2} If $e \wedge f = 0$, then $x_e + x_f = x_e \vee x_f$.
\item \label{Specker facts: 3} $x_e \vee x_{\lnot e} = 1$ and $x_e \wedge x_{\lnot e} = 0$.
\item \label{Specker facts: 4} $x_e = x_{e \wedge \lnot f} + x_{e \wedge f}$.
\end{enumerate}

In addition, if $a \in \mathbb{R}[B]$, we may write $a = r_1 x_{b_1} + \cdots + r_n x_{b_n}$ for some $r_i \in \mathbb{R}$ and $b_i \in B$ with $b_i \wedge b_j = 0$ whenever $i \ne j$ (see, e.g., \cite[Lem.~5.4]{BMO13a}). 
\end{remark}

We conclude this section by the following remark in which we collect together some well-known identities that hold in $\bal$-algebras. They will be used throughout the paper. Most can be found in \cite[Ch.~XIII, XV, XVII]{Bir79}, for (\ref{l-ring: vee wedge a}), (\ref{l-ring: inequality}), and (\ref{l-ring: join times scalar}) see \cite[Secs.~12, 13]{LZ71}, and for (\ref{l-ring: multiplication}) see \cite[Lem.~1]{Joh65}.

\begin{remark} \label{rem: l-ring properties}
Let $A \in \bal$, $a, b \in A$, and $C, D \subseteq A$. 
\begin{enumerate}
\item \label{l-ring: vee + a} If $\bigvee C$ exists in $A$, then $\bigvee \{ a + c \mid c \in C\}$ exists and is equal to $a + \bigvee C$. The dual property for meets also holds. 
\item \label{l-ring: vee wedge a} $A$ is a distributive lattice. 
Furthermore, if $\bigvee C$ exists in $A$, then $\bigvee \{ a \wedge c \mid c \in C\}$ exists and is equal to $a \wedge \bigvee C$. The dual property for meets also holds.  
\item \label{l-ring: negation and join} $-(a \vee b) = (-a) \wedge (-b)$ and $-(a \wedge b) = (-a) \vee (-b)$. 
\item \label{l-ring: -a plus} $a^- = (-a)^+$. 
\item \label{l-ring: positive plus negative} $a = a^+ - a^-$ and $|a| = a^+ + a^-$.
\item \label{l-ring: inequality} $(a + b)^+ \le a^+ + b^+$ and $(a+b)^- \le a^- + b^-$.
\item \label{l-ring: plus meet minus} $a^+ \wedge a^- = 0 = a^+ a^-$. 
\item \label{l-ring: join times scalar} If $\bigvee C$ exists in $A$ and $0 \le r \in \mathbb{R}$, then $\bigvee \{ rc \mid c \in C\}$ exists and is equal to $r\bigvee C$. The dual property for meets also holds. 
\item \label{l-ring: disjoint}If $a \wedge b = 0$, then $ra \wedge sb = 0$ for any $0 \le r, s \in \mathbb{R}$. 
\item \label{l-ring: multiplication} If $C, D$ consist of nonnegative elements and the joins $\bigvee C, \bigvee D$ exist in $A$, then $\left(\bigvee C\right) \left(\bigvee D\right) = \bigvee \{ cd \mid c \in C, d \in D \}$. The dual property for meets also holds.
\item \label{l-ring: preservation of abs value} If $\alpha$ is a $\bal$-morphism, then $\alpha(|a|) = |\alpha(a)|$. \end{enumerate}
\end{remark}

\section{Archimedean $\ell$-ideals} \label{sec: arch}

In this section we discuss archimedean $\ell$-ideals in $\bal$-algebras. This notion will be our main tool in the rest of the paper. 
Let $A\in\bal$ and $I$ be an $\ell$-ideal of $A$. It is well known and easy to check that $A/I$ is a bounded $\ell$-algebra, but $A/I$ may not be archimedean in general.

\begin{definition}
 Let $A \in \bal$. We call an $\ell$-ideal $I$ of $A$ \emph{archimedean} if $A/I$ is archimedean.  Let $\arch(A)$ be the set of archimedean $\ell$-ideals of $A$, ordered by inclusion.
\end{definition}

\begin{remark} \label{rem: properties of arch}
Let $A\in\bal$.
\begin{enumerate}
\item \label{properties of arch: 1} If $I$ is an $\ell$-ideal of $A$, then $A/I$ is archimedean iff $A/I\in\bal$. Thus, $I$ is archimedean iff $A/I \in \bal$. 
\item \label{properties of arch: 2} If $M$ is a maximal $\ell$-ideal of $A$, then it is well known (see, e.g., \cite[Cor.~27]{HJ61}) that $A/M \cong \mathbb{R}$. Thus, every maximal $\ell$-ideal is archimedean. 
\item \label{properties of arch: 3} Assuming (AC), an $\ell$-ideal $I$ of $A \in \bal$ is archimedean iff $I = \bigcap \{ M \in Y_A \mid I \subseteq M\}$ (see, e.g., \cite[p.~440]{BMO13a}).
\end{enumerate}
\end{remark}

\begin{remark}
In \cite{Ban97} Banaschewski studied the $\ell$-ideals in $f$-rings that are closed in the norm topology. If $A$ is a $\bal$-algebra, then it turns out that an $\ell$-ideal $I$ of $A$ is archimedean iff it is closed in the norm topology. We skip the proof since this fact does not play an important role in the rest of the paper. 
\end{remark}

The next result is a consequence of Banaschewski's more general result. 

\begin{theorem} \cite[App.~2]{Ban97}
\label{thm: arch frame}
If $A\in\bal$, then $\arch(A)$ is a compact regular frame.
\end{theorem}

\begin{remark} \label{ref: arch iso to opens}
Since $\arch(A)$ is compact regular, using (AC) it follows from \cite{Isb72} that $\arch(A)$ is isomorphic to the frame of open subsets of a compact Hausdorff space (see also \cite{BM80} or \cite[Sec.~III.1]{Joh82}). In fact, $\arch(A)$ is isomorphic to the frame of open subsets of the Yosida space $Y_A$ via the map that sends $I \in \arch(A)$ to the open set $Z_\ell(I)^c$. This is an analogue of the well-known fact that for a commutative ring $R$, there is a bijection between the frame of radical ideals of $R$ and the frame of open subsets of the prime spectrum of $R$ with the Zariski topology.
\end{remark}

It is straightforward to see that the intersection of a family of archimedean $\ell$-ideals is archimedean. Thus, we can define the concept of the archimedean hull, which Banaschewski \cite{Ban05a} referred to as the archimedean kernel.

\begin{definition}
Let $A\in\bal$. The \emph{archimedean hull} of $S\subseteq A$ is the intersection of all archimedean $\ell$-ideals of $A$ containing $S$. We denote the archimedean hull of $S$ by $\ar{S}$.
\end{definition}

Banaschewski \cite[p.~321]{Ban05a} showed that if $A$ is an $f$-ring and $I$ is an $\ell$-ideal of $A$, then the archimedean hull of $I$ is constructed as follows. Let
\[
k(I) = \{ a \in A \mid (n|a|-b)^+ \in I \textrm{ for some } b \ge 0 \textrm{ and for all }n \ge 1\}.
\]
It is straightforward to see that $k(I)$ is an $\ell$-ideal containing $I$. Set $k_1(I) = k(I)$. For each ordinal $\gamma$ set $k_{\gamma + 1}(I) = k(k_{\gamma}(I))$. If $\gamma$ is a limit ordinal, define $k_{\gamma}(I) = \bigcup \{ k_{\delta}(I) \mid \delta < \gamma\}$. Then there is a least $\gamma$ with $k_{\gamma}(I) = k_{\gamma + 1}(I)$, and the archimedean hull $\ar{I}$ of $I$ is $k_{\gamma}(I)$. 

We conclude this section by showing that when $A \in \bal$ (and hence 1 is a strong order-unit), the construction of the archimedean hull simplifies. For this we need the following remark.
 
\begin{remark}
\label{rem: positive part}
\begin{enumerate}
\item[]
\item \label{positive part: 1} If $A$ is a bounded $\ell$-algebra, then it is archimedean iff for each $a\in A$, whenever $na \le 1$ for each $n \ge 1$, then $a \le 0$. To see that this condition implies that $A$ is archimedean, suppose $na \le b$ for each $a,b\in A$ and $n\ge 1$. Since $A$ is bounded, there is $m \ge 1$ with $b \le m$. Therefore, $(nm)a \le b \le m$, so $na \le 1$ for all $n \ge 1$. This forces $a \le 0$, and hence $A$ is archimedean. 
\item \label{positive part: 2} Let $A\in\bal$ and $I$ be an $\ell$-ideal of $A$. For $a \in A$, we have: 
\[
a + I \ge 0 + I \mbox{ iff } a^- \in I \mbox{ and } a + I \le 0 + I \mbox{ iff } a^+ \in I \mbox{ (see, e.g., \cite[Rem.~2.11]{BMO16}).}
\]
\end{enumerate}
\end{remark}
 
\begin{proposition} \label{prop: archimedean hull}
Let $A\in\bal$ and $I$ be an $\ell$-ideal of $A$. Then
\[
\ar{I} = \{ x \in A \mid (n|x|-1)^+ \in I \textrm{ for all } n \ge 1 \}.
\] 
\end{proposition}

\begin{proof}
Set $K = \{ x \in A \mid (n|x|-1)^+ \in I \textrm{ for all } n \ge 1 \}$. It is straightforward to see that $K$ is an $\ell$-ideal of $A$ containing $I$. To see that $A/K$ is archimedean, by Remark~\ref{rem: positive part}(\ref{positive part: 1}), it is enough to show that if $a \in A$ with $n(a+K) \le 1 + K$ for each $n \ge 1$, then $a + K \le 0 + K$. Since $na + K \le 1 + K$ we get $na^+ +K \le 1 + K$. By Remark~\ref{rem: positive part}(\ref{positive part: 2}) this implies that $(na^+-1)^+ \in K$ for each $n \ge 1$. To show $a + K \le 0 + K$ it is sufficient to show that $a^+ \in K$. Thus, we may replace $a$ by $a^+$ to assume $0 \le a$. Since $(na-1)^+ \in K$ for each $n \ge 1$, we have $(m(na-1)^+-1)^+ \in I$ for each $n, m \ge 1$. Using several properties from Remark~\ref{rem: l-ring properties}, we have
\begin{align*}
(m(na-1)^+ - 1)^+ &= (((mna - m) \vee 0) - 1) \vee 0 \\
&= ((mna - (m+1)) \vee -1) \vee 0 \\
&= (mna - (m+1)) \vee 0 \\
&= (mna - (m+1))^+,
\end{align*}
and so $(mna - (m+1))^+ \in I$ for each $m, n \ge 1$. Set $n = m + 1$. Then $(m+1)(ma-1)^+ \in I$, so $(ma-1)^+ \in I$ because $m+1$ is a unit in $A$. Since this is true for each $m \ge 1$, we get $a \in K$, so $a + K = 0 + K$. Therefore, $A/K$ is archimedean, and hence $K$ is an archimedean $\ell$-ideal of $A$. 

Suppose $J \supseteq I$ is an archimedean $\ell$-ideal. If $a \in K$, then $(n|a|-1)^+ \in J$ for each $n \ge 1$, so $(n|a| - 1) + J \le 0 + J$, and hence $n|a| + J \le 1 + J$ for each $n \ge 1$. Since $A/J$ is archimedean, $|a| + J = 0 + J$, so $|a| \in J$, and hence $a \in J$. Therefore, $K \subseteq J$, and thus $K$ is the least archimedean $\ell$-ideal containing $I$, completing the proof.
\end{proof}

\section{Canonical extensions in \texorpdfstring{$\bal$}{bal} point-free} \label{sec: canonical extensions in bal point-free}

Canonical extensions of $\bal$-algebras were introduced in \cite{BMO18c}, where it was shown that a canonical extension of $A\in\bal$ is isomorphic to $(B(Y_A),\zeta_A)$. The proof that $(B(Y_A),\zeta_A)$ is a canonical extension of $A\in\bal$ is neither choice-free nor point-free. However, the uniqueness part of the proof is point-free. 
In this section we give a point-free proof of the existence as well, generalizing our results from Section~\ref{sec: canonical extensions in BA}. The arguments are considerably more complicated than those of Section~\ref{sec: canonical extensions in BA} and require a careful study of various properties of archimedean $\ell$-ideals. To make the section easier to read we collect all these properties together in an appendix.

\begin{definition}\cite[Def.~1.6]{BMO18c}
Let $A$ be a $\bal$-algebra, $D$ a Dedekind $\bal$-algebra and $\zeta : A \to D$ a $\bal$-monomorphism.
\begin{enumerate}
\item We call $\zeta$ \emph{compact} if whenever $S, T \subseteq A$ and $\varepsilon > 0$ with $\bigwedge \zeta[S] + \varepsilon \le \bigvee \zeta[T]$, there are finite $S_0 \subseteq S$ and $T_0 \subseteq T$ with $\bigwedge S_0 \le \bigvee T_0$.
\item We call $\zeta$ \emph{dense} if each element of $D$ is a join of meets from $\zeta[A]$.
\item We say that the pair $(D,\zeta)$ is a \emph{canonical extension} of $A$ if $\zeta$ is dense and compact.
\end{enumerate}
\end{definition}

\begin{remark} \label{rem: compact}
\begin{enumerate}
\item[]
\item In \cite{BMO18c} canonical extensions were defined for the category of bounded archimedean vector lattices, but the same definition works for $\bal$.
\item \label{compact: join} By \cite[Lem.~2.4]{BMO18c}, $\zeta : A \to D$ is compact iff for each $T \subseteq A$ and $\varepsilon > 0$, if $\varepsilon \le \bigvee \zeta[T]$, then there is a finite $T_0 \subseteq T$ with $0 \le \bigvee T_0$.
\end{enumerate}
\end{remark}

Let $A\in\bal$. By \cite[Thm.~1.8(2)]{BMO18c}, $(B(Y_A),\zeta_A)$ is a canonical extension of $A$. 
To motivate our new approach, we give a point-free description of $B(Y_A)$. Let ${\sf Dist}$ be the category of bounded distributive lattices. It is well known that the forgetful functor ${\mathcal U}:\BA\to{\sf Dist}$ has a left adjoint $\BE:{\sf Dist}\to\BA$. For each $L\in{\sf Dist}$, the Boolean algebra $\BE(L)$ together with the canonical embedding $i:L\to\BE(L)$ is usually referred to as the \emph{free Boolean extension} of $L$. It is characterized by the following universal mapping property: If $\lambda : L \to C$ is a bounded lattice homomorphism into a boolean algebra, then there is a unique $\BA$-morphism $\tau : \BE(L) \to C$ with $\tau \circ i = \lambda$ (see, e.g., \cite[Sec.~V.4]{BD74}). 
\[
\begin{tikzcd}
L \arrow[r, "i"] \arrow[dr, "\lambda"'] & \BE(L) \arrow[d, "\tau"]\\
& C
\end{tikzcd}
\]

Let $L = \arch(A)$. Viewing $L$ as a bounded distributive lattice, let $(B,i)$ be the free boolean extension of $L$. 
Assuming (AC), $L$ is isomorphic to the frame $\mathcal{O}(Y_A)$ of open sets of $Y_A$ (see Remark~\ref{ref: arch iso to opens}) and $B$ is isomorphic to the algebra ${\sf Con}(Y_A)$ of constructible sets of $Y_A$ (see, e.g., \cite[Sec.~V.4]{BD74}), where ${\sf Con}(Y_A)$ is the boolean subalgebra of $\wp(Y_A)$ generated by $\mathcal{O}(Y_A)$.  The isomorphism $\lambda : L \to \mathcal{O}(Y_A)$ sends $I$ to $Z_\ell(I)^c := \{ M \in Y_A \mid I \not\subseteq M\}$ (see Remark~\ref{ref: arch iso to opens}). Since $\mathcal{O}(Y_A)$ is a sublattice of $\wp(Y_A)$, we view $\lambda$ as a bounded lattice homomorphism into the boolean algebra $\wp(Y_A)$.
As $\wp(Y_A)$ is isomorphic to $\Id(B(Y_A))$ by sending $U\subseteq Y_A$ to the characteristic function $\chi_U$, the universal mapping property mentioned above yields a $\BA$-morphism $\tau : B \to \Id(B(Y_A))$ with $\tau(i(I)) = \chi_{Z_\ell(I)^c}$. By Theorem~\ref{thm: UMP}(\ref{UMP: 1}), there is a $\bal$-morphism $\sigma : \mathbb{R}[B] \to B(Y_A)$ extending $\tau$.
Since each $0 \le f \in B(Y_A)$ is equal to $\bigvee \{ f(M)\chi_{\{M\}} \mid M \in Y_A\}$,
each element of $B(Y_A)$ is a join from $\mathbb{R}[B]$ by Remark~\ref{rem: l-ring properties}(\ref{l-ring: vee + a}). Thus, by Theorem~\ref{thm: Ded}, there is an isomorphism $\theta : D(\mathbb{R}[B]) \to B(Y_A)$ satisfying $\theta(x_I) = \sigma(x_I) = \chi_{Z_\ell(I)^c}$ for each $I \in \arch(A)$. 
Let $\alpha = \theta^{-1} \circ \zeta_A$. 
\begin{figure}[h] 
\[
\begin{tikzcd}[row sep = 1pc]
L \arrow[r, "i"] \arrow[dd, "\lambda"'] & B \arrow[r] \arrow[dd, "\tau"'] & \mathbb{R}[B] \arrow[r, hookrightarrow] \arrow[ddr, "\sigma"'] & D(\mathbb{R}[B]) \arrow[dd, "\theta"]  &\\
&&&& A\arrow[ul, "\alpha"'] \arrow[dl, "\zeta_A"] \\
\mathcal{O}(Y_A) \arrow[r] & \func{Id}(B(Y_A)) \arrow[rr, hookrightarrow] && B(Y_A)  &
\end{tikzcd}
\]
\caption{The isomorphism $\theta : D(\mathbb{R}[B]) \to B(Y_A)$}\label{fig}
\end{figure}

If $0 \le a \in A$, we claim that
\[
\zeta_A(a) = \bigvee \{ r\chi_C \mid C \textrm{ closed in } Y_A \mbox{ and } r\chi_C \le \zeta_A(a) \}.
\]
It is obvious that $\zeta_A(a)$ is above the join. Let $M \in Y_A$ and set $\zeta_A(a)(M) = r$. Then $r\chi_{\{M\}}(M) = r$ and $r\chi_{\{M\}} \le \zeta_A(a)$ since $0 \le a$. From this it follows that the equation above is true. 

To make this description of $\zeta_A(a)$ point-free, let $C$ be closed in $Y_A$ and $I \in \arch(A)$ with $Z_\ell(I) = C$. We show that $r\chi_C \le \zeta_A(a)$ for some $r \ge 0$ iff $(a-r)^- \in I$.

First suppose $r\chi_C \le \zeta_A(a)$ for some $r \ge 0$. 
If $M \in C$, then $r \le \zeta_A(a)(M)$, so $a + M \ge r + M$, which means $(a-r)^- \in M$ (see Remark~\ref{rem: positive part}(\ref{positive part: 2})). Since this is true for all $M \in C$, we have $(a-r)^- \in \bigcap Z_\ell(I) = I$ (see Remark~\ref{rem: properties of arch}(\ref{properties of arch: 3})).

Conversely, let $(a-r)^- \in I$ and $M \in C$. Then $I \subseteq M$, so $(a-r)^- \in M$ which gives $a + M \ge r + M$, so $r \le \zeta_A(a)(M)$. Since $0 \le a$, this yields $r\chi_C \le \zeta_A(a)$. Consequently,
\begin{align*}
\zeta_A(a) &= \bigvee \{ r\chi_C \mid C \textrm{ closed in } Y_A \mbox{ and } r\chi_C \le \zeta_A(a) \} \\
&= \bigvee \{ r\chi_{Z_\ell(I)} \mid I \in \arch(A) \mbox{ and } (a-r)^- \in I\}.
\end{align*}

For $a \in A$ arbitrary, let $s \in \mathbb{R}$ be such that $a + s \ge 0$. Then
\begin{align*}
\zeta_A(a) &= -s + \zeta_A(a+s) \\
&= -s + \bigvee \{ r\chi_{Z_\ell(I)} \mid I \in \arch(A), (a + s -r)^- \in I\}.
\end{align*}

\begin{definition}
Let $(B,i)$ be the free boolean extension of $\arch(A)$. For ease of notation we assume that $\arch(A) \subseteq B$, and so $i$ is the identity. Thus, for $I \in \arch(A)$ we have that $x_I$ and $x_{\lnot I}$ are idempotents of $\mathbb{R}[B]$.
\end{definition}

The discussion above motivates the following point-free definition.

\begin{definition}\label{def: point-free alpha}
Define $\alpha : A \to D(\mathbb{R}[B])$ by
\[
\alpha(a) = -s + \bigvee \{ rx_{\lnot I} \mid  I \in \arch(A), (a + s - r)^- \in I \}
\]
where $s \in \mathbb{R}$ with $a + s \ge 0$.
\end{definition}
 
 \begin{remark} \label{rem: r ge 0}
 Let $a \in A$ and $s \in \mathbb{R}$ with $a + s \ge 0$. Then $(a + s)^- = 0$. Thus, for each $I \in \arch(A)$, we have that $0x_{\lnot I}$ is part of the join defining $\alpha(a)$. Consequently, we may assume $r \ge 0$ in the definition of $\alpha(a)$.
 \end{remark}
 
To show that $\alpha$ is well defined we need to show that the join in Definition~\ref{def: point-free alpha} exists and that the expression does not depend on the choice of $s$. Showing that the join exists is straightforward, but independence of $s$  
requires some work. In particular, we will utilize Lemma~\ref{lem: joins in D} given in the appendix.

\begin{proposition} \label{prop: alpha is well-defined}
$\alpha : A \to D(\mathbb{R}[B])$ is well defined.
\end{proposition}

\begin{proof}
Let $a \in A$ and $s \in \mathbb{R}$ with $a + s \ge 0$. We first show that 
\[
\{ r x_{\lnot I} \mid I \in \arch(A), (a+s-r)^- \in I\}
\] 
is bounded above, so the join exists in $D(\mathbb{R}[B])$. Let $I \in \arch(A)$. If $I = A$, then $x_{\lnot I} = 0$, so $rx_{\lnot I} = 0$. Suppose that $I \ne A$. If $r > \|a\| + s + 1$, then $a + s - r \le \|a\| + s - r < -1$, so $(a + s -r)^- > 1$. Therefore, $(a + s -r)^- \notin I$. Thus, if $(a + s -r)^- \in I$, then $r \le \|a\| + s + 1$, and so $rx_{\lnot I} \le \|a\| + s +1$ as $x_{\lnot I}$ is an idempotent. From this it follows that the set above is bounded by $\|a\| + s + 1$, and hence the join defining $\alpha(a)$ exists.

We next show that $\alpha(a)$ does not depend on $s$. Let $0 \le s, t \in \mathbb{R}$ with $a + s, a + t \ge 0$. Set $f = \bigvee \{ rx_{\lnot I} \mid I \in \arch(A), (a + s - r)^- \in I\}$ and $g = \bigvee \{ rx_{\lnot I} \mid I \in \arch(A), (a + t - r)^- \in I\}$. Then $f, g \in D(\mathbb{R}[B])$ by the previous paragraph. By Lemma~\ref{lem: joins in D}(\ref{joins in D: 3}) and Remark~\ref{rem: l-ring properties}(\ref{l-ring: vee + a}),
\begin{align*}
f + t &= \bigvee \{ (r + t)x_{\lnot I} \mid (a+s-r)^- \in I\} \\
&= \bigvee \{ ux_{\lnot I} \mid (a + s + t - u)^- \in I\} \\
&= \bigvee \{ (v+s)x_{\lnot I} \mid (a + t - v)^- \in I\} \\
&= s + g.
\end{align*}
Therefore, $-s + f = -t + g$, which proves that the formula defining $\alpha(a)$ does not depend on the choice of $s$. Thus, $\alpha$ is well defined.
\end{proof}

We next show that $\alpha$ preserves order and addition by a scalar.

\begin{lemma} \label{lem: preservation of adding a scalar}
\begin{enumerate}
\item[]
\item \label{preservation of adding a scalar: 1} If $0 \le a\in A$, then $\alpha(a) = \bigvee \{ rx_{\lnot I} \mid 0 \le r, I \in \arch(A), (a-r)^- \in I\}$.
\item \label{preservation of adding a scalar: 2} $\alpha$ is order preserving.
\item \label{preservation of adding a scalar: 3} If $a\in A$ and $t\in \mathbb{R}$, then $\alpha(a + t) = \alpha(a) + t$.
\end{enumerate}
\end{lemma}

\begin{proof}
(\ref{preservation of adding a scalar: 1}) Set $s = 0$ and apply Remark~\ref{rem: r ge 0}.

(\ref{preservation of adding a scalar: 2}) Suppose that $a \le b$. Choose $0 \le s$ with $a + s \ge 0$. Then $b + s \ge 0$. Therefore, $(a+s-r)^- \in I$ implies $(b+s-r)^- \in I$ because $(b+s-r)^- \le (a+s-r)^-$. From this it follows that $\alpha(a) \le \alpha(b)$.

(\ref{preservation of adding a scalar: 3}) Let $s_1,s_2\in\mathbb{R}$ be such that $a + s_1 \ge 0$ and $t + s_2 \ge 0$. Set $s = s_1 + s_2$. By Remark~\ref{rem: r ge 0},
\[
\alpha(a) = -s_1 + \bigvee \{ rx_{\lnot I} \mid 0 \le r, I \in \arch(A), (a + s_1 - r)^- \in I\},
\] 
so by Lemma~\ref{lem: joins in D}(\ref{joins in D: 3}), we have 
\begin{align*}
\alpha(a) + s_2 + t &= -s_1 + \bigvee \{ (s_2 + t + r)x_{\lnot I} \mid (a + s_1 - r)^- \in I\} \\
&= -s_1 + \bigvee \{ (s_2 + t + r)x_{\lnot I} \mid (a + s + t - (s_2 + t + r))^- \in I\} \\
&= s_2 -s + \bigvee \{ ux_{\lnot I} \mid (a + s + t - u)^- \in I\} \\
&= s_2 + \alpha(a + t)
\end{align*}
\color{black}since $a+t+s \ge 0$\color{black}. Thus, $\alpha(a + t) = \alpha(a) + t$.
\end{proof}

We are ready to prove that $(D(\mathbb{R}[B]),\alpha)$ is 
a canonical extension of $A$. This we do in the next three propositions.

\begin{proposition}
\label{prop: monomorphism}
$\alpha:A\to D(\mathbb{R}[B])$ is a $\bal$-monomorphism.
\end{proposition}

\begin{proof}
We first show that $\alpha$ is a $\bal$-morphism. This we do by showing that $\alpha$ preserves addition, meet, scalar multiplication, join, and multiplication.

\begin{claim} \label{claim: preserves +}
$\alpha$ preserves addition. 
\end{claim}

\begin{proofclaim}
Let $a, b \in A$. We first assume that $0 \le a, b$.
Because $\alpha$ is order preserving (see Lemma~\ref{lem: preservation of adding a scalar}(\ref{preservation of adding a scalar: 2})) and $a,b \le a+b$, we have that $\alpha(a), \alpha(b) \le \alpha(a+b)$. In addition, by Lemma~\ref{lem: preservation of adding a scalar}(\ref{preservation of adding a scalar: 1}) (and Remarks~\ref{rem: l-ring properties}(\ref{l-ring: vee + a},\ref{l-ring: disjoint}) and~\ref{rem: Specker facts}(\ref{Specker facts: 2})),
\begin{align*}
\alpha(a) + \alpha(b) &= \bigvee \{ rx_{\lnot I} \mid (a-r)^- \in I\} + \bigvee \{ sx_{\lnot J} \mid (b-s)^- \in J\} \\
&= \bigvee \{ rx_{\lnot I} + sx_{\lnot J} \mid (a-r)^- \in I, (b-s)^- \in J \} \\
&= \bigvee \{ rx_{\lnot I \wedge J} + sx_{\lnot J \wedge I} + (r+s) x_{\lnot I \wedge \lnot J} \mid (a-r)^- \in I, (b-s)^- \in J \} \\
&= \bigvee \{ rx_{\lnot I \wedge J} \vee sx_{\lnot J \wedge I} \vee (r+s) x_{\lnot (I \vee J)} \mid (a-r)^- \in I, (b-s)^- \in J \}.
\end{align*}
We use this to show that $\alpha(a) + \alpha(b) \le \alpha(a+b)$. We have $rx_{\lnot I \wedge J} \le rx_{\lnot I} \le \alpha(a) \le \alpha(a+b)$ and $sx_{\lnot J \wedge I} \le sx_{\lnot J} \le \alpha(b) \le \alpha(a+b)$. It remains to show that $(r+s)x_{\lnot (I \vee J)} \le \alpha(a+b)$. This is trivial if $I \vee J = A$ since then $x_{\lnot (I \vee J)} = 0$. Otherwise, set $K = I \vee J$. We have $(a-r)^-, (b-s)^- \in K$, and since $0 \le (a+b - (r+s))^- \le (a-r)^- + (b-s)^-$ by Remark~\ref{rem: l-ring properties}(\ref{l-ring: inequality}), we see that $(a+b-(r+s))^- \in K$. Therefore, $(r+s)x_{\lnot K} \le \alpha(a+b)$ by Lemma~\ref{lem: preservation of adding a scalar}(\ref{preservation of adding a scalar: 1}). This completes the proof that $\alpha(a) + \alpha(b) \le \alpha(a+b)$.

We use Lemma~\ref{lem: f below g} to show the reverse inequality. Suppose that $tx_{\lnot I} \le \alpha(a + b)$. We may assume that $t = \sup \{ r \mid rx_{\lnot I} \le \alpha(a+b) \}$. Then $J := I \vee \ar{(a+b-t)^+} \ne A$ and $(a+b-t)^- \in I$ by Lemmas~\ref{lem: arch} and \ref{lem: not A}, so $(a+b-t)^+ - (a+b - t)^- \in J$, and hence $a + b - t \in J$ by Remark~\ref{rem: l-ring properties}(\ref{l-ring: positive plus negative}). Let $s = \sup\{ r \mid rx_{\lnot J} \le \alpha(a) \}$. Then $sx_{\lnot J} \le \alpha(a)$, $(a-s)^- \in J$, and $K := J \vee \ar{(a-s)^+} \ne A$, again by Lemmas~\ref{lem: arch} and \ref{lem: not A}. We have $a-s \in K$, so $b-(t-s) = (a + b - t) - (a - s) \in K$. Therefore, $(b - (t-s))^- \in K$, so $(t-s) x_{\lnot K} \le \alpha(b)$ by Lemma~\ref{lem: preservation of adding a scalar}(\ref{preservation of adding a scalar: 1}). Since $J \subseteq K$, we have $sx_{\lnot K} \le sx_{\lnot J} \le \alpha(a)$. From this we see that $tx_{\lnot K} \le \alpha(a) + \alpha(b)$. Consequently, by Lemma~\ref{lem: f below g}, $\alpha(a+b) \le \alpha(a) + \alpha(b)$. This shows that $\alpha(a + b) = \alpha(a) + \alpha(b)$ for $0 \le a, b$. 

To complete the argument, let $a,b \in A$ be arbitrary and choose $t \in \mathbb{R}$ with $a+t, b+t \ge 0$. Then $\alpha(a+b+2t) = \alpha(a+t) + \alpha(b + t)$, so $\alpha(a+b)+2t = \alpha(a) + \alpha(b) + 2t$ by Lemma~\ref{lem: preservation of adding a scalar}(\ref{preservation of adding a scalar: 3}). Therefore, $\alpha(a+b) = \alpha(a) + \alpha(b)$.
\end{proofclaim}

\begin{claim}
$\alpha$ preserves meet.
\end{claim}

\begin{proofclaim}
Let $a,b \in A$. Since $\alpha$ is order preserving, 
$\alpha(a \wedge b) \le \alpha(a) \wedge \alpha(b)$. Because $\alpha$ preserves addition (by a scalar), we may assume $0 \le a,b$. By Lemmas~\ref{lem: preservation of adding a scalar}(\ref{preservation of adding a scalar: 1}) and~\ref{rem: re meet sf}(\ref{re meet sf: 1}) (and Remarks~\ref{rem: l-ring properties}(\ref{l-ring: vee wedge a}) and \ref{rem: Specker facts}(\ref{Specker facts: 1})), we have 
\begin{align} \label{eqn: meet}
\begin{split}
\alpha(a) \wedge \alpha(b) &= \bigvee \{ rx_{\lnot I} \mid  (a-r)^- \in I\} \wedge \bigvee \{ sx_{\lnot J} \mid (b-s)^- \in J \} \\
&= \bigvee \{ rx_{\lnot I} \wedge sx_{\lnot J} \mid (a-r)^- \in I, (b-s)^- \in J \} \\
&= \bigvee \{ \min(r, s) x_{\lnot I \wedge \lnot J} \mid (a-r)^- \in I, (b-s)^- \in J \} \\
&= \bigvee \{ \min(r, s) x_{\lnot (I \vee J)} \mid (a-r)^- \in I, (b-s)^- \in J \}.
\end{split}
\end{align}
Suppose $(a-r)^- \in I$, $(b-s)^- \in J$, and assume without loss of generality that $r \le s$. Since $(b-r)^- \le (b-s)^-$, we have $(a-r)^-, (b-r)^- \in I \vee J$. Therefore, $(a-r)^- \vee (b-r)^- \in I \vee J$ (because $\ell$-ideals are closed under $\vee$). Since
\begin{align*}
(a-r)^- \vee (b-r)^- &= [(r-a)\vee 0] \vee [(r-b)\vee 0] = [(r-a) \vee (r-b)]\vee 0 \\
&= [r + (-a \vee -b)] \vee 0 = [r - (a \wedge b)] \vee 0 = ((a \wedge b) - r)^-,
\end{align*}
we see that $((a\wedge b) - r)^- \in I \vee J$. Therefore, $rx_{\lnot(I \vee J)} \le \alpha(a \wedge b)$, and hence (\ref{eqn: meet}) implies that $\alpha(a) \wedge \alpha(b) \le \alpha(a \wedge b)$. Thus, $\alpha(a) \wedge \alpha(b) = \alpha(a \wedge b)$.
\end{proofclaim}

\begin{claim}
$\alpha$ preserves scalar multiplication and join.
\end{claim}

\begin{proofclaim}
Since $\alpha$ is a group homomorphism, to show that it preserves scalar multiplication it suffices to show $\alpha(sa) = s\alpha(a)$ for each $a \in A$ and $0 < s \in \mathbb{R}$. Moreover, by Lemma~\ref{lem: preservation of adding a scalar}(\ref{preservation of adding a scalar: 3}), it suffices to assume $0 \le a$. Since $(sa-r)^- \in I$ iff $(a-r/s)^- \in I$, by Lemma~\ref{lem: preservation of adding a scalar}(\ref{preservation of adding a scalar: 1}) and Remark~\ref{rem: l-ring properties}(\ref{l-ring: join times scalar}), we have
\begin{align*}
\alpha(sa) &= \bigvee \{ rx_{\lnot I} \mid (sa - r)^- \in I \} = \bigvee \{ rx_{\lnot I} \mid (a - r/s)^- \in I \} \\
&= s\bigvee \{ (r/s)x_{\lnot I} \mid (a - r/s)^- \in I \} = s\alpha(a).
\end{align*}
Since $\alpha$ preserves meet and scalar multiplication, it preserves join by Remark~\ref{rem: l-ring properties}(\ref{l-ring: negation and join}).
\end{proofclaim}

\begin{claim}
$\alpha$ preserves multiplication.
\end{claim}

\begin{proofclaim}
First suppose that $0 \le a, b$. By Lemma~\ref{lem: preservation of adding a scalar}(\ref{preservation of adding a scalar: 1}) and Remarks~\ref{rem: l-ring properties}(\ref{l-ring: multiplication}) and \ref{rem: Specker facts}(\ref{Specker facts: 1}), we have
\begin{align*}
\alpha(a) \alpha(b) &= \bigvee \{ rx_{\lnot I} \mid 0 \le r, (a-r)^- \in I\} \cdot \bigvee \{ sx_{\lnot J} \mid 0 \le s, (b-s)^- \in J\} \\
&= \bigvee \{ rsx_{\lnot I} x_{\lnot J} \mid (a-r)^- \in I, (b-s)^- \in J\} \\
&= \bigvee \{ rsx_{\lnot(I \vee J)} \mid (a-r)^- \in I, (b-s)^- \in J\}.
\end{align*}
Also $\alpha(ab) = \bigvee \{ tx_{\lnot K} \mid (ab-t)^- \in K \}$. To see that $\alpha(a)\alpha(b) \le \alpha(ab)$, it suffices to show that $rsx_{\lnot{(I \vee J)}} \le \alpha(ab)$, where $0 \le r, s$, $(a-r)^- \in I$, and $(b-s)^- \in J$. Set $K = I \vee J$. Then $(a-r)^-, (b-s)^- \in K$. Therefore, $r + K \le a + K$ and $s + K \le b + K$. Because $0 \le r, s$, we have
\[
rs + K = (r+K)(s+K) \le (a+K)(b+K) = ab + K.
\]
This gives $(ab-rs)^- \in K$. Thus, $rsx_{\lnot(I \vee J)} = rsx_{\lnot K} \le \alpha(ab)$ by Lemma~\ref{lem: preservation of adding a scalar}(\ref{preservation of adding a scalar: 1}). Consequently, 
$\alpha(a)\alpha(b) \le \alpha(ab)$.

For the reverse inequality, we use Lemma~\ref{lem: f below g} and argue as in the proof of Claim~\ref{claim: preserves +}. Let $I \in \arch(A)$ and $0 \le t$ with $tx_{\lnot I} \le \alpha(ab)$. Then $(ab - t)^- \in I$ by Lemma~\ref{lem: not A}(\ref{not A: 2}). We may assume that $t = \sup\{ r \mid (ab-r)^- \in I\}$ by Lemma~\ref{lem: not A}(\ref{not A: 1}). Therefore, $J := I \vee \ar{(ab-t)^+} \ne A$ by Lemma~\ref{lem: not A}(\ref{not A: 3}). Set $r = \sup\{ p \mid (a-p)^- \in J\}$. Then $(a-r)^- \in J$ and $K := J \vee \ar{(a-r)^+} \ne A$. We have $ab - t \in J$ and $a-r \in K$. Therefore, $ab - t, (a-r)b \in K$. Thus, $rb - t \in K$. If $r = 0$, then $t \in K$, which implies $t = 0$ since $K \ne A$ and nonzero real numbers are units in $A$. It is then clear that $tx_{\lnot K} \le \alpha(a)\alpha(b)$ since $0 \le \alpha(a)\alpha(b)$. If $r \ne 0$, then $b - t/r \in K$. Therefore, $(t/r)x_{\lnot K} \le \alpha(b)$ and $rx_{\lnot K} \le \alpha(a)$, so $tx_{\lnot K} \le \alpha(a)\alpha(b)$ since $x_{\lnot K}$ is an idempotent. Thus, by Lemma~\ref{lem: f below g}, $\alpha(ab) \le \alpha(a)\alpha(b)$. This shows that $\alpha(ab) = \alpha(a)\alpha(b)$ for $0 \le a, b$.

For $a,b$ arbitrary, since $a = a^+ - a^-$ and $b = b^+ - b^-$ (see Remark~\ref{rem: l-ring properties}(\ref{l-ring: positive plus negative})), we have
\[
ab = (a^+ - a^-)(b^+ - b^-) = (a^+b^+ + a^- b^-) - (a^+ b^- + a^- b^+)
\]
By the previous case and Claim~\ref{claim: preserves +},
\begin{align*}
\alpha(ab) &= \alpha(a^+b^+ + a^- b^-) - \alpha(a^+ b^- + a^- b^+) \\
&= (\alpha(a^+ b^+) + \alpha(a^- b^-)) - (\alpha(a^+ b^-) + \alpha(a^- b^+)) \\
&= (\alpha(a^+)\alpha(b^+) + \alpha(a^-)\alpha(b^-)) - (\alpha(a^+)\alpha(b^-) + \alpha(a^-)\alpha(b^+)) \\
&= (\alpha(a^+) - \alpha(a^-))(\alpha(b^+) - \alpha(b^-)) = \alpha(a^+ - a^-)\alpha(b^+ - b^-)\\
&= \alpha(a)\alpha(b).
\end{align*}
\end{proofclaim}

This completes the proof that $\alpha$ is a $\bal$-morphism.
It is left to show that $\alpha$ is a monomorphism. This we do by showing that $\alpha(a) \ne 0$ for $a \ne 0$. Because $\alpha(|a|) = |\alpha(a)|$ (see Remark~\ref{rem: l-ring properties}(\ref{l-ring: preservation of abs value})) and $a \ne 0$ iff $|a| \ne 0$, it suffices to assume $a \ge 0$. Since $a \ne 0$, we have $\|a\| > 0$. Choose $r\in\mathbb{R}$ with $0 < r \le \|a\|$ and set $I = \ar{(a-r)^-}$. We claim that $I \ne A$. To see this, if $I = A$, by Lemma~\ref{lem: arch}(\ref{arch: 2}) there is $n$ with $1 \le n(a-r)^-$, so $1/n \le (a-r)^- = (r-a)\vee 0$. Therefore, $1/n \le r-a$ by Lemma~\ref{rem: re meet sf}(\ref{re meet sf: 3}). This implies $a \le r - 1/n < r$, which contradicts the inequality $r \le \|a\|$ since $\|a\|$ is the smallest real number above $a$. Thus, $\lnot I \ne 0$, so $x_{\lnot I} \ne 0$, and hence $0 < rx_{\lnot I} \le \alpha(a)$. This shows that $\alpha(a) \ne 0$.
\end{proof}

\begin{proposition}\label{prop: dense}
$\alpha:A\to D(\mathbb{R}[B])$ is dense.
\end{proposition}

\begin{proof}
We first show that $x_{\lnot I}$ is a meet from $\alpha[A]$ for each $I \in \arch(A)$. By Lemma~\ref{lem: preservation of adding a scalar}(\ref{preservation of adding a scalar: 1}), if $(a-1)^- \in I$, then $x_{\lnot I} \le \alpha(a)$, so  
\[
x_{\lnot I} \le \bigwedge \{ \alpha(a) \mid (a-1)^- \in I\} =: f. 
\]
In the inequality above we may assume that $0\le a\le 1$ since $x_{\lnot I}$ is an idempotent, so $0 \le x_{\lnot I}$ and $x_{\lnot I} \le \alpha(a)$ implies that $x_{\lnot I} \le \alpha(a) \wedge 1 =\alpha(a\wedge 1)$. Therefore, we have $0 \le f \le 1$. If $I \subseteq J$, then $x_{\lnot J} \le x_{\lnot I}$. Because $0 \le f$, it is a join of elements of the form $rx_{\lnot J}$ by Lemma~\ref{lem: joins in D}(\ref{joins in D: 2}). To show $f = x_{\lnot I}$ we show that if $rx_{\lnot J} \le f$ with $r > 0$, then $J \supseteq I$ and $r \le 1$, so $rx_{\lnot J} \le x_{\lnot I}$. First, if $rx_{\lnot J} \le f$, then $rx_{\lnot J} \le 1$, so $r \le 1$ by Remark~\ref{rem: idempotent facts}(\ref{idempotent facts: 1}). Next, suppose that $I \not\subseteq J$. Then there is $K \ne A$ with $J \subseteq K$ and $K + I = A$ by Lemma~\ref{lem: arch}(\ref{arch: 2}). Therefore, there are $0 \le a,b$ with $a \in K$, $b \in I$, and $1 = a + b$ by Lemma~\ref{lem: I = A}(\ref{I = A: 1}). So $(a-1)^- = b \in I$, and hence $x_{\lnot I} \le \alpha(a)$. Suppose that there is $r > 0$ with $(a-r)^- \in J$. Then $A = \ar{a,(a-r)^-}$ by Lemma~\ref{lem: I = A}(\ref{I = A: 1}). This is impossible since $K \ne A$ but $a, (a-r)^- \in K$. This implies that $r \le 0$. This contradiction shows that $x_{\lnot I} = f$ and hence is a meet from $\alpha[A]$. 

If $0 \le f \in D(\mathbb{R}[B])$, then $f$ is a join of nonnegative elements from $\mathbb{R}[B]$. If $0 \le c \in \mathbb{R}[B]$, we can write $c$ as a sum of terms of the form $rx_b$ with $0 \le r \in \mathbb{R}$ and $b \in B$, and so $f$ is a join of such terms by Remark~\ref{rem: Specker facts}. Since each $b$ can be written as a join of terms of the form $\lnot I \wedge J$ with $I, J \in \arch(A)$, we see that $f$ is a join of elements of the form $rx_{\lnot I\wedge J}$. Therefore, by Lemma~\ref{lem: joins in D}(\ref{joins in D: 1}), $f$ is a join of elements of the form $rx_{\lnot I}$. Thus, $f$ is a join of meets from $\alpha[A]$. For $f$ arbitrary, if $f + n \ge 0$, then $f + n$ is a join of meets from $\alpha[A]$, and so $f$ is also a join of meets from $\alpha[A]$ by Remark~\ref{rem: l-ring properties}(\ref{l-ring: vee + a}). Consequently, $\alpha$ is dense.
\end{proof}

\begin{proposition} \label{prop: compact}
$\alpha : A \to D(\mathbb{R}[B])$ is compact.
\end{proposition}

\begin{proof}
Let $0 < \varepsilon \in \mathbb{R}$ and $T \subseteq A$ with $\varepsilon \le \bigvee \alpha[T]$. By Remark~\ref{rem: compact}(\ref{compact: join}), it suffices to show that there is a finite $T_0 \subseteq T$ with $\bigvee T_0 \ge 0$. Set $T' = \{ (a + \varepsilon) \vee 0 \mid a \in T\}$. Since $\alpha((a+\varepsilon)\vee 0) \ge \alpha(a + \varepsilon) = \alpha(a) + \varepsilon$, we have $2\varepsilon \le \bigvee \alpha[T']$ by Remark~\ref{rem: l-ring properties}(\ref{l-ring: vee + a}).
As $0 \le b$ for each $b \in T'$, Lemma~\ref{lem: preservation of adding a scalar}(\ref{preservation of adding a scalar: 1}) implies
\[
2\varepsilon \le \bigvee \{ rx_{\lnot I} \mid (b - r)^- \in I, b \in T' \}.
\]
We next consider the archimedean $\ell$-ideal $L = \bigvee \{ \ar{(b - \varepsilon)^+} \mid b \in T'\}$ and show that $L = A$. If not, then $x_{\lnot L} \ne 0$, so using Remarks~\ref{rem: Specker facts}(\ref{Specker facts: 1}) and \ref{rem: l-ring properties}(\ref{l-ring: multiplication}), we have
\begin{align} \label{eqn}
\begin{split}
2\varepsilon x_{\lnot L} &\le x_{\lnot L} \cdot \bigvee \{ rx_{\lnot I} \mid (b - r)^- \in I, b \in T' \} \\
&= \bigvee \{ rx_{\lnot L} x_{\lnot I} \mid (b - r)^- \in I, b \in T' \} \\
&= \bigvee \{ rx_{\lnot (I \vee L)}\mid (b - r)^- \in I, b \in T' \}.
\end{split}
\end{align}
Observe that if $I \vee L = A$, then $x_{\lnot(I \vee L)} = 0$. Suppose that $r \le 3\varepsilon/2$ for all $r, b, I$ in the join above with $I \vee L \ne A$. Because $x_{\lnot (I \vee L)} \le x_{\lnot L}$ for each $I$, the join above is then bounded by $(3\varepsilon/2)x_{\lnot L}$, a contradiction to the inequality (\ref{eqn}). Therefore, there are $r, b, I$ in the join above with  $I \vee L \ne A$ and $r > 3\varepsilon/2$. Since $(b-3\varepsilon/2)^- \le (b-r)^-$, we have $(b-3\varepsilon/2)^- \in I \subseteq I \vee L$, and $(b-\varepsilon)^+ \in L \subseteq I \vee L$ by definition of $L$. Therefore, $I \vee L = A$ by Lemma~\ref{lem: I = A}(\ref{I = A: 1}). This contradiction yields $L = A$. 

Since $\arch(A)$ is a compact frame (see Theorem~\ref{thm: arch frame}), there are $b_1, \dots, b_n \in T'$ with $\ar{(b_1 - \varepsilon)^+} \vee \cdots \vee \ar{(b_n - \varepsilon)^+} = A$. For each $i$ there is $a_i \in T$ with $b_i = (a_i + \varepsilon) \vee 0$. Then, using Remark~\ref{rem: l-ring properties}(\ref{l-ring: vee + a}), we have
\[
(b_i - \varepsilon)^+ =  [((a_i + \varepsilon) \vee 0) - \varepsilon] \vee 0 = (a_i \vee -\varepsilon) \vee 0 = a_i \vee 0 = a_i^+
\]
for each $i$. Therefore, $\ar{a_1^+} \vee \cdots \vee \ar{a_n^+} = A$. Set $c = a_1 \vee \cdots \vee a_n$. Then $a_i^+ \le c^+$ for each $i$, so $\ar{c^+} = A$. Thus, there is $m \ge 1$ with $1 \le mc^+$, and hence $1/m \le c^+$. By Lemma~\ref{rem: re meet sf}(\ref{re meet sf: 3}), $1/m \le c$ which yields $0 \le 1/m \le a_1 \vee \cdots \vee a_n$. This shows that $\alpha$ is compact.
\end{proof}

Propositions~\ref{prop: monomorphism},~\ref{prop: dense}, and \ref{prop: compact} yield our main result.

\begin{theorem} \label{thm: main}
For each $A\in\bal$, the pair $(D(\mathbb{R}[B]),\alpha)$ is a canonical extension of $A$.
\end{theorem}

\section{Canonical extensions and normal functions} \label{sec: normal}

In the previous section we gave a point-free description of a canonical extension of $A\in\bal$ as the pair $(D(\mathbb{R}[B]),\alpha)$. In this section we show that $D(\mathbb{R}[B])$ can be described as the algebra $N(X)$ of (bounded) normal real-valued functions on the space $X$ of proper archimedean $\ell$-ideals of $A$. The idempotents of $N(X)$ are exactly the characteristic functions of regular opens of $X$. Thus, we obtain a generalization of a result of \cite{BH20} that a canonical extension of a boolean algebra $B$ is isomorphic to the boolean algebra of regular open subsets of the space of proper filters of $B$. Assuming (AC), we show that $N(X)$ is isomorphic to the algebra of bounded real-valued functions on the Yosida space of $A$, thus obtaining a result of \cite{BMO18c}. We conclude the paper by drawing a connection between our results and those in point-free topology describing normal functions on an arbitrary frame  \cite{GKP09,GP11,GP14,GMP16}.

For a topological space $X$, we recall that $B(X)$ is the set of all bounded real-valued functions on $X$. It is straightforward to see that under pointwise operations $B(X) \in \dbal$. Recall that $f \in B(X)$ is 
\emph{lower semicontinuous} if $f^{-1}(r, \infty)$ is open, and $f$ is \emph{upper semicontinuous} if $f^{-1}(-\infty, r)$ is open for each $r \in \mathbb{R}$ (see, e.g., \cite[p.~361]{Bou89}). 
For each $x \in X$, let ${\mathcal N}_x$ be the collection of open neighborhoods of $x$. For each $f \in B(X)$ 
define
\begin{align*}
f_*(x) &= \sup \{ \inf f[U] \mid U \in \mathcal{N}_x \} \\
f^*(x) &= \inf \{ \sup f[U] \mid U \in \mathcal{N}_x \}.
\end{align*}
It is well known that $f$ is lower semicontinuous iff $f = f_*$ and $f$ is upper semicontinuous iff $f = f^*$ (see, e.g., \cite[p.~360-362]{Bou89}). 

Since we will be interested in the poset and the corresponding Alexandroff space of proper archimedean $\ell$-ideals, we will utilize the following lemma. 

\begin{lemma} \label{lem: order preserving}
Let $X$ be an Alexandroff space and $f \in B(X)$. 
\begin{enumerate}
\item \label{order preserving: 1} $f$ is lower semicontinuous iff $f$ is order preserving.
\item \label{order preserving: 2} $f$ is upper semicontinuous iff $f$ is order reversing.
\end{enumerate}
\end{lemma}

\begin{proof}
We only prove (\ref{order preserving: 1}) as (\ref{order preserving: 2}) is proved similarly. First suppose that $f$ is order preserving. Let $s \in \mathbb{R}$. Since $(s, \infty)$ is an upset in $\mathbb{R}$ and $f$ is order preserving, $f^{-1}(s, \infty)$ is an upset in $X$. Therefore, $f^{-1}(s, \infty)$ is open in $X$. Thus, $f$ is lower semicontinuous. 

Conversely, suppose that $f$ is lower semicontinuous. Let $x,y \in X$ with $x \le y$. If $f(x) = s$, then for each $r < s$, we have $x \in f^{-1}(r,\infty)$, which is an open subset of $X$ since $f$ is lower semicontinuous. Therefore, $y \in f^{-1}(r, \infty)$. Thus, for each $r \in \mathbb R$ we have $r < f(x)$ implies $r < f(y)$. This forces $f(x) \le f(y)$, and hence $f$ is order preserving.
\end{proof}

\begin{remark}
It is well known that a map between Alexandroff spaces is continuous iff it is order preserving. Therefore, $f \in B(X)$ is lower semicontinuous iff $f$ is continuous with respect to the Alexandroff topology on $\mathbb R$, and $f$ is upper semicontinous iff $f$ is continuous with respect to the topology of downsets of $\mathbb{R}$.
\end{remark}

The following definition is motivated by Dilworth \cite{Dil50}.

\begin{definition} 
Let $X$ be a topological space and $f \in B(X)$. We call $f^\# := (f^*)_*$ the \emph{normalization} of $f$ and we call $f$ \emph{normal} if $f = f^\#$. Let 
\[
N(X) = \{ f \in B(X) \mid f = f^\# \}.
\]
\end{definition}

\begin{theorem} \label{thm: danet}
$N(X) \in \dbal$ and the operations on $N(X)$ are normalizations of the corresponding operations on $B(X)$.
\end{theorem}

\begin{proof}
It follows from \cite{Dil50} that $N(X)$ is a Dedekind complete lattice where bounded joins and meets are normalizations of pointwise bounded joins and meets. By \cite{Dan15}, $N(X)$ is a lattice-ordered vector space, where addition and scalar multiplication are normalizations of pointwise addition and scalar multiplication. Finally, by \cite[Sec.~8]{BMO16}, $N(X)\in\bal$, where multiplication is the normalization of pointwise multiplication. Thus, $N(X) \in \dbal$. We point out that in \cite{Dil50,Dan15} $X$ is assumed to be completely regular and in \cite{BMO16} compact Hausdorff, but the same proofs work for an arbitrary topological space.
\end{proof}

We next show that idempotents of $N(X)$ correspond to regular opens of $X$.

\begin{lemma}\label{lem: idempotents}
For a topological space $X$, the idempotents of $N(X)$ are precisely the characteristic functions $\chi_U$ for $U$ a regular open subset of $X$. Consequently, $\Id(N(X)) \cong \RO(X)$.
\end{lemma}

\begin{proof}
Let $e \in N(X)$ be an idempotent. Then $e = 2e \wedge 1$ since $(2e \wedge 1) - e = e \wedge (1-e) = 0$. Since positive scalar multiplication and meet in $N(X)$ are pointwise (see, e.g., \cite[Thm.~5.1]{Dan15}), the equation $e = 2e \wedge 1$ yields that $e(x) \in \{0, 1\}$ for each $x \in X$. Therefore, $e$ is a characteristic function.
If $e = \chi_U$ for $U \subseteq X$, then 
it is straightforward to see that $e^* = \chi_{\Cl(U)}$ and $e_* = \chi_{\Int(U)}$. Thus, $e^\# = \chi_{\Int\,\Cl(U)}$, and so $e = e^\#$ iff $U$ is regular open in $X$. It is then straightforward to check that the map $U \mapsto \chi_U$ is an order preserving and order reflecting bijection, and hence a boolean isomorphism between $\RO(X)$ and $\Id(N(X))$.
\end{proof}

For a poset $X$ and $S \subseteq X$, we use the standard notation
\begin{align*}
\up S &= \{ x \in S \mid s \le x \textrm{ for some } s \in S\} \\
\down S &= \{ x \in S \mid x \le s \textrm{ for some } s \in S\}.
\end{align*} 
If $S = \{x\}$ is a singleton, we write $\up x$ for $\up S$ and $\down x$ for $\down S$. The closure and interior operators of the Alexandroff topology on $X$ are given by
\[
\Cl(S) = \down S \textrm{\ \ and\  \ }\Int(S) = \{ x \in X \mid \up x \subseteq S\}.
\]

\begin{lemma}\label{lem: more idempotents}
Let $A\in\bal$, $X= \arch(A) \setminus \{ A \}$, and $I \in X$. Then $\up I$ and $U_I := \{ J \in X \mid J \vee I = A\}$ are regular open subsets of $X$, and $U_I$ is the complement of $\up I$ in $\RO(X)$.
\end{lemma}

\begin{proof}
Since $\up I$ is an upset, hence open in $X$, the inclusion $\up I \subseteq \Int\,\Cl (\up I)$ is clear. For the reverse inclusion, suppose that $J \notin \up I$. By Lemma~\ref{lem: I = A}(\ref{I = A: 2}), there is $K \supseteq J$ with $K + I = A$. Therefore, $K \notin \down\up I$, so $\up J \not\subseteq \down\up I = \Cl(\up I)$, showing that $J \notin \Int\,\Cl(\up I)$.  Thus, $\up I \in \RO(X)$. 

Since $U_I$ is an upset, the inclusion $U_I \subseteq \Int\,\Cl(U_I)$ is clear. Suppose that $J \notin U_I$. Then $K := J \vee I \ne A$. If $K \subseteq L$ with $L \vee I = A$, then $L = A$ since $I \subseteq K \subseteq L$. Therefore, $\up K \not\subseteq \down U_I$, so $\up J \not\subseteq \Cl(U_I)$, and hence $J \not\in \Int\,\Cl(U_I)$. This shows $U_I = \Int\,\Cl(U_I)$, so $U_I \in \RO(X)$. 

Finally, to see that $U_I$ is the complement of $\up I$ in $\RO(X)$, it is clear that $\up I \cap U_I = \varnothing$. Let $V \in \RO(X)$ with $\up I \cap V = \varnothing$. If $J \in V$, then $I \not\subseteq J$. Therefore, by Lemma~\ref{lem: I = A}(\ref{I = A: 2}), there is $K \in X$ with $J \subseteq K$ and $K + I = A$, so $K \vee I = A$ as $K + I \subseteq K \vee I$. Thus, $K \in U_I$, and hence $J \in \down U_I$. This shows that $V \subseteq \down U_I$, so $V \subseteq \Int\,\Cl(U_I) = U_I$. Consequently, $U_I$ is the complement of $\up I$ in $\RO(X)$.
\end{proof}

From now on we will assume that $X$ is the set of proper archimedean $\ell$-ideals of $A\in\bal$ ordered by inclusion. The proof of the next theorem is choice-free.

\begin{theorem} \label{thm:normal}
There is a $\bal$-isomorphism $\varphi : D(\mathbb{R}[B]) \to N(X)$ such that $\varphi(x_I) = \chi_{U_I}$ for each $I \in X$.
\end{theorem}

\begin{proof}
We first define $\lambda : \arch(A) \to \Id(N(X))$ by setting $\lambda(I) = \chi_{U_I}$. 
By Lemmas~\ref{lem: idempotents} and \ref{lem: more idempotents}, $\chi_{U_I} \in \Id(N(X))$, so $\lambda$ is well defined. We show that $\lambda$ is a bounded lattice homomorphism. It is clear that $U_0 = \varnothing$ and $U_A = X$. We show that $U_{I \cap J} = U_I \cap U_J$. 
It is obvious that $I \subseteq J$ implies $U_I \subseteq U_J$. Therefore, $U_{I \cap J} \subseteq U_I \cap U_J$. For the reverse inclusion, suppose that $K \in U_I \cap U_J$. Then $K \vee I = K \vee J = A$, so $(K \vee I) \cap (K \vee J) = A$.
Since $\arch(A)$ is a frame, $K \vee (I \cap J) = A$, and so $K \in U_{I \cap J}$. We next show that $U_{I \vee J} = U_I \vee U_J$. The inclusion $U_I \vee U_J \subseteq U_{I \vee J}$ is obvious. 
For the reverse inclusion, suppose that $K \in U_{I \vee J}$. Then $K \vee (I \vee J) = A$. Let $L \in X$ with $K \subseteq L$. If $L \vee I = A$, then $L \in U_I$. If not, then as $(L \vee I) \vee J = A$, we have $L \vee I \in U_J$ so $L \in \down U_J$. Therefore, in any case, 
$L \in \down U_I \cup \down U_J$, and so 
\[
K \in \Int(\down U_I \cup \down U_J) = \Int(\down (U_I \cup U_J)) = \Int\,\Cl(U_I \cup U_J) = U_I \vee U_J.
\] 
Thus, $\lambda$ is a bounded lattice homomorphism, and hence
it extends to a $\BA$-morphism $\tau : B \to \Id(N(X))$ (see, e.g., \cite[Sec.~V.4]{BD74}).  By Theorem~\ref{thm: UMP}(\ref{UMP: 1}), there is a $\bal$-morphism $\sigma : \mathbb{R}[B] \to N(X)$ with $\sigma(x_I) = \tau(I) = \chi_{U_I}$. 

To simplify notation, set $e_I = \chi_{\up I}$ and $f_I = \chi_{U_I}$. Then $e_I, f_I$ are complementary idempotents of $N(X)$ by Lemma~\ref{lem: more idempotents}. 
We show that $\sigma$ is one-to-one. Let $a \in \mathbb{R}[B]$ with $\sigma(a) = 0$. By Remark~\ref{rem: Specker facts}, we may write $a = r_1 x_{b_1} + \cdots + r_n x_{b_n}$ for some $r_i \in \mathbb{R}$ and $b_i \in B$ with $b_i \wedge b_j = 0$ whenever $i \ne j$. \color{black} From this we see that $ax_{b_i} = r_i x_{b_i} \in \ker(\sigma)$. Therefore, 
it suffices to show that $\sigma(rx_b) = 0$ implies $rx_b = 0$. If $r = 0$, this is clear, so suppose $r \ne 0$. Then $x_b \in \ker(\sigma)$. Since $B$ is generated by $\arch(A)$, we may write $b = (I_1 \wedge \lnot J_1) \vee \cdots \vee (I_n \wedge \lnot J_n)$ for some $I_k, J_k \in \arch(A)$. Then $0 \le x_{I_k \wedge \lnot J_k} \le x_b$, so each $x_{I_k \wedge \lnot J_k} \in \ker(\sigma)$.  Suppose that $\sigma(x_{I \wedge \lnot J}) = 0$. We have
\[
0 = \sigma(x_{I \wedge \lnot J}) = \sigma(x_I) \wedge \sigma(x_{\lnot J}) = \tau(I) \wedge \lnot \tau(J) = f_I \wedge e_J,
\]
where the last equality follows from Lemma~\ref{lem: more idempotents}. Therefore, $e_J \le \lnot f_I = e_I$, so $\up J \subseteq \up I$. This yields $I \subseteq J$, so $I \wedge \lnot J = 0$ in $B$, and hence $x_{I \wedge \lnot J} = 0$. This shows that $\sigma$ is one-to-one.

We next show that each element of $N(X)$ is a join from $\sigma[\mathbb{R}[B]]$. To do this we first show that each nonnegative element of $N(X)$ is a join of scalar multiples of the $e_I = \chi_{\up I}$. Let $0 \le f \in N(X)$. We show that $f$ is the join of those $re_I$ for $r \in \mathbb{R}$ with $re_I \le f$. Clearly $f$ is above this join. Let $I \in X$ and set $r = f(I)$. Then $f(J) \ge r$ for each $J \in \up I$ since $f$ is order preserving. Because $0 \le f$, this shows $re_I \le f$. But $(re_I)(I) = r$. Therefore, $\bigvee \{ re_I \mid r \in \mathbb{R}, re_I \le f\} = f$, and so $\bigsqcup \{ re_I \mid r \in \mathbb{R}, re_I \le f\} = f$, where $\bigsqcup$ is the normalization of $\bigvee$ (see Theorem~\ref{thm: danet}). 
For an arbitrary $f \in N(X)$, there is $s \in \mathbb{R}$ with $f + s \ge 0$. Thus, $f+s$ is a join from the image of $\mathbb{R}[B]$, and hence so is $f$ by Remark~\ref{rem: l-ring properties}(\ref{l-ring: vee + a}). Consequently, by Theorem~\ref{thm: Ded}, there is a $\bal$-isomorphism $\varphi : D(\mathbb{R}[B]) \to N(X)$ with $\varphi(x_I) = \sigma(x_I) = \chi_{U_I}$.
\end{proof}

\begin{theorem}\label{thm:normal & AC}
Assuming \emph{(AC)}, there is a $\bal$-isomorphism $\psi : N(X) \to B(Y_A)$ such that $\psi(f) = f|_{Y_A}$ is the restriction of $f$ to $Y_A$.
\end{theorem}

\begin{proof}
To see that $\psi$ is a $\bal$-morphism, we first observe that if $f \in B(X)$ and $I \in X$, then since $\up I$ is the least open neighborhood of $I$ in $X$, we have 
\[
f^*(I) = \inf \{ \sup f[U] \mid U \in \mathcal{N}_I \} = \sup\{ f(J) \mid I \subseteq J \}.
\]
This yields that $f^*(M)
=f(M)$ for each $M \in Y_A$. A similar calculation gives $f_*(I) = \inf\{ f(J) \mid I \subseteq J \}$.
Therefore, since $f^*$ is order reversing by Lemma~\ref{lem: order preserving}(\ref{order preserving: 2}),
\begin{align*}
f^\#(I) &=(f^*)_*(I)=\inf\{ f^*(J) \mid I \subseteq J \} =\inf\{ f^*(M) \mid M \in Y_A, I \subseteq M \} \\
& =\inf\{ f(M) \mid M \in Y_A, I \subseteq M \}.
\end{align*}
Consequently, $f^\#|_{Y_A}=f|_{Y_A}$.

Denoting the sum in $N(X)$ by $\oplus$, we have for $f,g \in N(X)$
\[
\psi(f\oplus g) = \psi((f+g)^\#) =  
(f+g)^\#|_{Y_A} =  
(f+g)|_{Y_A} =
 f|_{Y_A} + g|_{Y_A} = \psi(f) + \psi(g).
\]
A similar calculation shows that $\psi$ preserves the other operations. Thus, $\psi$ is a $\bal$-morphism.

We next show that $\psi$ is onto. Let $h \in B(Y_A)$ and define $h^u$ on $X$ by
\begin{align*}
h^u(I) &= \inf \{ h(M) \mid M \in Y_A, I \subseteq M\}. 
\end{align*}
Then $h^u \in B(X)$ and
\begin{align*}
(h^u)^\#(I)=\inf\{ h^u(M) \mid M \in Y_A, I \subseteq M \}=\inf\{ h(M) \mid M \in Y_A, I \subseteq M \}=h^u(I).
\end{align*}
This implies that $h^u \in N(X)$. By definition of $h^u$ we have $ \psi(h^u)=h^u|_{Y_A} = h$. Thus, $\psi$ is onto. 

Finally, we show that $\psi$ is one-to-one. Let $f,g \in N(X)$ with $\psi(f) = \psi(g)$. Then $f|_{Y_A}=g|_{Y_A}$ and 
for each $I \in X$ we have
\begin{align*}
f(I) &=f^\#(I)=\inf\{ f(M) \mid M \in Y_A, I \subseteq M \}=\inf\{ f|_{Y_A}(M) \mid M \in Y_A, I \subseteq M \}\\
&=\inf\{ g|_{Y_A}(M) \mid M \in Y_A, I \subseteq M \}=\inf\{ g(M) \mid M \in Y_A, I \subseteq M \}=g^\#(I)=g(I),
\end{align*}
which yields $f=g$. Thus, $\psi$ is one-to-one.
\end{proof}

Recalling the isomorphism $\theta : D(\mathbb{R}[B]) \to B(Y_A)$ in the beginning of Section~\ref{sec: canonical extensions in bal point-free} (see  Figure~\ref{fig}) and putting Theorems~\ref{thm:normal} and~\ref{thm:normal & AC} together, we obtain:

\begin{theorem}
Assuming \emph{(AC)}, for $A \in \bal$, the algebras $D(\mathbb{R}[B])$, $N(X)$, and $B(Y_A)$ are all isomorphic. Moreover, if $\gamma = \varphi \circ \alpha$, then $$\gamma(a)(I) = \sup\{ r \in \mathbb{R} \mid (a-r)^- \in I\}$$
and the following diagram commutes.
\[
\begin{tikzcd}[column sep = 5pc]
& D(\mathbb{R}[B]) \arrow[d, "\varphi"'] \arrow[dd, bend left = 40, "\theta"]\\
A \arrow[rd, "\zeta_A"'] \arrow[ru, "\alpha"] \arrow[r, "\gamma"] & N(X) \arrow[d, "\psi"'] \\
& B(Y_A) 
\end{tikzcd}
\]
\end{theorem}

\begin{proof} 
The isomorphism $\theta : D(\mathbb{R}[B]) \to B(Y_A)$ satisfies $\theta(x_I) = \chi_{Z_\ell(I)^c}$ for each $I \in \arch(A)$ and $\alpha = \theta^{-1} \circ \zeta_A$, where $Z_\ell(I)^c$ denotes the complement of $Z_\ell(I)$ in $Y_A$. 
We show that $\theta = \psi \circ \varphi$.  Since all three maps are $\dbal$-isomorphisms and so preserve arbitrary joins, it is enough to show that they agree on $\mathbb{R}[B]$. For the latter it is enough to show that they agree on each $x_I$ for $I \in X$. We have $\theta(x_{I}) = \chi_{Z_\ell(I)^c}$ and $\varphi(x_I) = \chi_{U_I}$. Since $\psi(\chi_{U_I}) = \chi_{U_I}|_{Y_A}$, we then need to show that $Z_{\ell}(I)^c = U_I \cap Y_A$. To see this, if $M \in Y_A$, then $M \in Z_{\ell}(I)^c$ iff $I \not\subseteq M$. Since $M$ is maximal, $I \not\subseteq M$ iff $M + I = A$. But $M + I = A$ iff $M \vee I = A$ by Lemma~\ref{lem: arch}(\ref{arch: 4}). Since $M \vee I = A$ iff $M \in U_I$, it follows that $Z_\ell(I)^c = U_I \cap Y_A$, which completes the proof that $\theta = \psi \circ \varphi$. Thus, 
\[
\psi \circ \gamma = \psi \circ \varphi \circ \alpha = \theta \circ \alpha = \zeta_A,
\]
which shows that the diagram is commutative.

It is left to show that the formula for $\gamma$ is valid. Suppose that $0 \le a \in A$. Since $\varphi$ preserves arbitrary joins, by Lemmas~\ref{lem: preservation of adding a scalar}(\ref{preservation of adding a scalar: 1}),~\ref{lem: more idempotents}, and Theorem~\ref{thm:normal}, we have
\begin{align*}
\gamma(a) &= \varphi(\alpha(a)) = \bigsqcup \{ \varphi(rx_{\lnot I}) \mid (a-r)^- \in I \} \\
&= \bigsqcup \{ r\varphi(x_{\lnot I}) \mid (a-r)^- \in I \} \\
&= \bigsqcup \{ r \chi_{\up I} \mid (a-r)^- \in I\}.
\end{align*}
Let $f \in B(X)$ be the pointwise join of $\{ r\chi_{\up I} \mid (a-r)^- \in I \}$. Then $\gamma(a) = f^\#$ by Theorem~\ref{thm: danet}. We claim that $f(J) = \sup\{ r \mid (a - r)^- \in J\}$ for each $J \in X$, and that $f \in N(X)$.
We have
\begin{align*}
f(J) &= 
\sup \{ r\chi_{\up I}(J) \mid (a-r)^- \in I \} \\
&= \sup \{ r \mid (a-r)^- \in I, J \in \up I \} \\
&= \sup \{ r \mid (a-r)^- \in J\}.
\end{align*}
To see the last equality, if $(a-r)^- \in I$ and $J \in \up I$, then $(a-r)^- \in J$. Conversely, if $(a-r)^- \in J$, then setting $I = J$, we have $(a-r)^- \in I$ and $J \in \up I$. 

To show that $f \in N(X)$, by \cite[Thm.~3.2]{Dil50} it is enough to show that $f$ is lower semicontinuous and $f^{-1}(-\infty, s)$ is a union of regular closed sets for each $s \in \mathbb{R}$. First, $f$ is clearly order preserving, so $f$ is lower semicontinuous by Lemma~\ref{lem: order preserving}(\ref{order preserving: 1}). Next, Let $I \in f^{-1}(-\infty, s)$. Set $t = f(I) < s$, so $(a-t)^- \in I$ by Lemma~\ref{lem: not A}(\ref{not A: 1}). In addition, $J := I \vee \ar{(a-t)^+} \ne A$ by Lemma~\ref{lem: not A}(\ref{not A: 3}). Let $U = \up J$, an open subset of $X$. We claim that $I \in \Cl(U) \subseteq f^{-1}(-\infty, s)$. Since $I \subseteq J$, we have $I \in \down(\up J) = \Cl(U)$. Because $f$ is order preserving, $f^{-1}(-\infty, s)$ is a downset. Therefore, to show that $\Cl(U) \subseteq f^{-1}(-\infty, s)$, it suffices to show that $\up J \subseteq f^{-1}(-\infty, s)$. Let $K \in X$ with $J \subseteq K$. We have $(a-t)^- \in I \subseteq J$, so $(a-t)^- \in K$, and $(a-t)^+ \in J$, so $(a-t)^+ \in K$.  Thus, $a-t \in K$ by Remark~\ref{rem: l-ring properties}(\ref{l-ring: positive plus negative}). Because $(a-t)^- \in K$, we have $t \le f(K)$.  Let $f(K) = r$. Then $(a-r)^- \in K$ by Lemma~\ref{lem: not A}(\ref{not A: 1}). If $r > t$, we have $\ar{(a-t)^+, (a-r)^-} = A$ by Lemma~\ref{lem: I = A}(\ref{I = A: 1}), so $K = A$. This contradiction shows that $f(K) \le t$, so $f(K) < s$. Therefore, $\up J \subseteq f^{-1}(-\infty, s)$, and hence $f^{-1}(-\infty, s)$ is a union of regular closed sets. Thus, we conclude that $f \in N(X)$. 

Since $f \in N(X)$, we have $f^\# = f$, so $\gamma(a) = f^\# = f$. 
This shows that if $0 \le a$, then $\gamma(a)(I) = \sup \{ r \mid (a-r)^- \in I\}$ for all $I \in X$. 
If $a$ is arbitrary, then there is $n \ge 1$ with $a + n \ge 0$. Since $\gamma$ preserves addition, by the above argument we have:
\begin{align*}
\gamma(a)(J) &= (\gamma(a+n) - n)(J) = \sup \{ r \mid (a + n -r)^- \in J\} - n \\
&= \sup \{ r - n \mid (a + n - r)^- \in J\} = \sup \{ s \mid (a - s)^- \in J \},
\end{align*}
completing the proof.
\end{proof}

Consequently, we have three equivalent ways to think about canonical extensions of $\bal$-algebras:
\begin{enumerate}
\item The simplest is as $(B(Y_A), \zeta_A)$ which is a direct generalization of viewing the powerset of its Stone space as a canonical extension of a boolean algebra \cite{JT51}. However, this requires (AC). 
\item A choice-free description is as $(N(X), \gamma)$, which generalizes the choice-free description of a canonical extension of a boolean algebra as regular opens of the Alexandroff space of its proper filters given in \cite{BH20}. 
\item Finally, a point-free description is as $(D(\mathbb{R}[B]), \alpha)$, which is technically the most challenging.
 It is this description that generalizes the point-free description of a canonical extension of a boolean algebra given in Section~\ref{sec: canonical extensions in BA}. 
\end{enumerate}

In point-free topology there is a description of normal functions on an arbitrary frame  \cite{GKP09,GP11,GP14,GMP16}. We finish the article by the following remark, which connects our results to that line of research. 

\begin{remark}
We recall that $\Frm$ is the category of frames and frame homomorphisms. If $L, K$ are frames, then we write $\hom_{\Frm}(L,K)$ for the set of frame homomorphisms from $L$ to $K$. Let $L$ be a frame. It is well known that homomorphic images of $L$ are characterized by nuclei on $L$ (see, e.g., \cite[Sec.~III.5.3]{PP12}). For a frame $L$ we write $\func{Nuc}(L)$ for the frame of nuclei on $L$. We also write $\L(\mathbb{R})$ for the frame of opens of $\mathbb{R}$. A point-free description of $\L(\mathbb{R})$ is due to Banaschewski \cite{Ban97} (see also \cite[Sec.~XIV.1]{PP12}). The role of $C(X)$ is then played by the $\ell$-algebra $\mathcal{C}(L) = \hom_{\Frm}(\L(\mathbb{R}), L)$. 

 As was shown in \cite[Sec.~5]{GKP09} the role of the algebra of all real-valued functions on $X$ is played by the $\ell$-algebra $F(L) = \hom_{\Frm}(\L(\mathbb{R}), \func{Nuc}(L))$, and that of $B(X)$ by the bounded subalgebra $F^*(L)$ of $F(L)$. Then the operators $(-)^*, (-)_* : B(X) \to B(X)$ generalize to $(-)^*, (-)_* : F^*(L) \to F^*(L)$ \cite[Sec.~3]{GMP16}, yielding the notion of normal function on $L$. We write $N(L) = \{ f \in F^*(L) \mid (f^*)_* = f\}$ for the set of (bounded) normal functions on $L$. It follows from \cite{Dan15} and \cite[Sec.~8]{BMO16} that $N(L) \in \dbal$. 
 
 Let $A \in \bal$,  $L =\arch(A)$, and $(B, i)$ be the free boolean extension of $L$. Then $B$ is isomorphic to $\Id(N(L))$, yielding that $D(\mathbb{R}[B])$ is isomorphic to $N(L)$. Thus, our point-free description of a canonical extension of $A$ can alternatively be described using the algebra of normal functions in point-free topology.
 \end{remark}

\appendix

\section*{Appendix: Technical lemmas}
\renewcommand{\thetheorem}{A.\arabic{theorem}}
\setcounter{theorem}{0}

In this appendix we present the technical lemmas used in Section~\ref{sec: canonical extensions in bal point-free} to prove that the pair $(D(\mathbb{R}[B]),\alpha)$ is a canonical extension of $A\in\bal$.
We start by recalling that $0 \le u \in A$ is a \emph{weak order-unit} if $a \wedge u = 0$ implies $a = 0$ for each $a \in A$. It is well known that a strong order-unit is a weak order-unit (see, e.g., \cite[Lem.~XIII.11.4]{Bir79}). It is easy to see that any positive multiple of a strong order-unit is again a strong order-unit. Thus, every positive multiple of a strong order-unit is a weak order-unit. We will use this in the proof of next lemma.

\begin{lemma} \label{rem: re meet sf}
Let $A \in \bal$.
\begin{enumerate}
\item \label{re meet sf: 1} If $e, f \in \Id(A)$, and $0 \le r, s \in \mathbb{R}$, then $re \wedge sf = \min(r, s)(e \wedge f)$. 
\item \label{re meet sf: 3} Let $a \in A$ and $r, s \in \mathbb{R}$ with $r < s$. If $a \vee r \ge s$, then $a \ge s$.
\end{enumerate}
\end{lemma}

\begin{proof}

(\ref{re meet sf: 1}) Without loss of generality suppose that $r \le s$. Since $(e - (e\wedge f)) \wedge f = (e - f) \wedge f= 0$, we have $r(e - (e\wedge f))\wedge sf = 0$ by Remark~\ref{rem: l-ring properties}(\ref{l-ring: disjoint}). Therefore,
\[
0 \le (re \wedge sf) - r(e \wedge f) \le r(e - (e\wedge f))\wedge sf = 0
\]
by Remark~\ref{rem: l-ring properties}(\ref{l-ring: vee + a}). Thus, $re \wedge sf = r(e \wedge f) = \min(r,s)(e \wedge f)$.

(\ref{re meet sf: 3}) If $a \vee r \ge s$, then $-s \ge -(a \vee r) = (-a) \wedge (-r)$, so $0 \ge s + [(-a) \wedge (-r)] = (s-a) \wedge (s-r)$. This yields
\[
0 = [(s-a) \wedge (s-r)] \vee 0 = [(s-a) \vee 0] \wedge (s-r).
\]
Because $s-r > 0$, it is a weak order-unit. 
Therefore, $(s-a) \vee 0 = 0$, and so $s-a \le 0$. Thus, $s \le a$.
\end{proof}

\begin{lemma}  \label{lem: arch}
Let $A \in \bal$.
\begin{enumerate}
\item \label{arch: 1} If $I + J = A$, then there are $a \in I$, $b \in J$ with $0 \le a, b$ and $a + b = 1$.
\item \label{arch: 2} $\ar{I} = A$ implies $I = A$.
\item \label{arch: 3}If $I, J \in \arch(A)$, then $I \vee J = \ar{I + J}$.
\item \label{arch: 4} If $I,J \in \arch(A)$ and $I \vee J = A$, then $I + J = A$.
\end{enumerate}
\end{lemma}

\begin{proof}
(\ref{arch: 1}) Since $I + J = A$, there are $x \in I$ and $y \in J$ with $1 = x + y$. Since $y \in J$, we have $1 + J = x + J \le |x| + J$. Set $a = 1 \wedge |x|$. Then $0 \le a \le 1$ and $a \in I$ because $x \in I$, so $|x| \in I$. Therefore, $a + J = (1 + J) \wedge (|x| + J) = 1 + J$. Thus, $b := 1-a \in J$. Clearly $a + b = 1$ and $0 \le b$ since $a \le 1$.

(\ref{arch: 2}) Suppose $\ar{I} = A$. Then $1 \in \ar{I}$, so $(n\cdot 1-1)^+ \in I$ for each $n \ge 1$ by Proposition~\ref{prop: archimedean hull}. In particular, $(2\cdot 1-1)^+ \in I$. Thus, $1 \in I$, and so $I = A$.

(\ref{arch: 3}) Since $I + J \subseteq I \vee J$ and $I \vee J$ is archimedean, we have $\ar{I+J} \subseteq I \vee J$. On the other hand, $\ar{I+J}$ is an archimedean ideal which contains both $I$ and $J$, so it contains $I \vee J$. Thus, $I \vee J = \ar{I + J}$.

(\ref{arch: 4}) This follows from (\ref{arch: 2}) and (\ref{arch: 3}).
\end{proof}

Let $L$ be a frame and $B$ its free boolean extension. We recall that for $a \in L$, we write $a^*$ for the pseudocomplement of $a$ in $L$. On the other hand, we write $\lnot a$ for the complement of $a$ in $B$.

\begin{lemma} \label{lem: joins in D}
Let $A \in \bal$ and $B$ be the free boolean extension of $\arch(A)$. 
\begin{enumerate}
\item \label{joins in D: 1} If $I \in \arch(A)$, then $x_I = \bigvee \{ x_{\lnot J} \mid J \vee I = A\}$. 
\item \label{joins in D: 2} If $0 \le f \in D(\mathbb{R}[B])$, then there are $0 \le r_I \in \mathbb{R}$ with $f = \bigvee \{r_I x_{\lnot I} \mid I \in \arch(A) \}$.
\item \label{joins in D: 3} Let $0 \le f \in D(\mathbb{R}[B])$ and $0 \le t \in \mathbb{R}$. If $f = \bigvee \{ r_I x_{\lnot I} \mid I \in \arch(A) \}$ with $r_I \ge 0$, then $f + t = \bigvee \{ (t + r_I)x_{\lnot I} \mid I \in \arch(A)\}$.
\end{enumerate}
\end{lemma}

\begin{proof}
(\ref{joins in D: 1}) Since $\arch(A)$ is a regular frame (see Theorem~\ref{thm: arch frame}),
\[
I = \bigvee \{ K \mid K \prec I\} = \bigvee \{ K \mid K^* \vee I = A\}. 
\]
We show that $I = \bigvee \{ \lnot J \mid J \vee I = A\}$. The right-to-left inclusion is clear. For the left-to-right inclusion it is sufficient to show that if $K^* \vee I = A$, then $K \le \lnot J$ for some $J$ with $J \vee I = A$. But if we set $J = K^*$, then $K \subseteq K^{**} = J^* \le \lnot J$. Thus, $I = \bigvee \{ \lnot J \mid J \vee I = A\}$, and hence $x_I = \bigvee \{ x_{\lnot J} \mid J \vee I = A\}$.

(\ref{joins in D: 2}) Each element of $D(\mathbb{R}[B])$ is a join from $\mathbb{R}[B]$. A nonnegative element of $\mathbb{R}[B]$ can be written in the form $r_1 x_{b_1} + \cdots + r_n x_{b_n} = r_1 x_{b_1} \vee \cdots \vee r_n x_{b_n}$ for some $0 \le r_i \in \mathbb{R}$ and $b_1, \dots, b_n \in B$ with $b_i \wedge b_j = 0$ whenever $i \ne j$ (see Remark~\ref{rem: Specker facts}). Since each $b \in B$ is a finite join of elements of the form $J \wedge \lnot I$ with $I, J \in \arch(A)$, we may write a nonnegative element of $\mathbb{R}[B]$ as a join of elements of the form $r(x_J \wedge x_{\lnot I})$. Thus, by (\ref{joins in D: 1}), if $0 \le f \in D(\mathbb{R}[B])$, we may write $f$ as a join of elements of the form $rx_{\lnot I}$ with $I \in \arch(A)$.

(\ref{joins in D: 3})
By Remark~\ref{rem: Specker facts}(\ref{Specker facts: 2},\ref{Specker facts: 3}), $x_I + x_{\lnot I} = x_I \vee x_{\lnot I} = 1$. Therefore, by Remark~\ref{rem: l-ring properties}(\ref{l-ring: vee + a}), 
\[
f + t = \bigvee \{ t + r_I x_{\lnot I} \mid I \in \arch(A) \} = \bigvee \{ (t + r_I)x_{\lnot I} + tx_I \mid I \in \arch(A) \}. 
\]
Because $x_{\lnot I} \wedge x_I = 0$, Remark~\ref{rem: l-ring properties}(\ref{l-ring: disjoint}) implies $(t + r_I)x_{\lnot I} \wedge tx_I = 0$, so $(t + r_I)x_{\lnot I} + tx_I = (t + r_I)x_{\lnot I} \vee tx_I$. Thus, by (\ref{joins in D: 1}),
\begin{align*}
f + t &= \bigvee \{  (t + r_I)x_{\lnot I} \vee tx_I \mid I \in \arch(A) \} \\
&= \bigvee \{ (t + r_I)x_{\lnot I} \vee tx_{\lnot J} \mid I, J \in \arch(A), I \vee J = A\}.
\end{align*}
Now, $t \le t + r_{J}$, so $tx_{\lnot J} \le (t+r_{J})x_{\lnot J}$. Consequently, $f + t = \bigvee \{  (t + r_I)x_{\lnot I} \mid I \in \arch(A)\}$.
\end{proof}

\begin{remark} \label{rem: [S]}
Let $A\in\bal$. For $S\subseteq A$ we let $\el{S}$ be the $\ell$-ideal of $A$ generated by $S$. It is well known (see, e.g., \cite[p.~96]{LZ71}) that
\[
\el{S} = \{ x \in A : |x| \le n|a| \textrm{ for some } n \ge 1, a \in S\}.
\]
If $S = \{a\}$, we write $\el{a}$ for $\el{S}$.
\end{remark}

\begin{lemma} \label{lem: I = A}
Let $A \in \bal$ and let $X = \arch(A) \setminus \{A\}$.
\begin{enumerate}
\item \label{I = A: 1} If $a, b \in A$ with $a < b$ and $b - a \in \mathbb{R}$, then $\ar{b^+, a^-} = A$.
\item \label{I = A: 2} If $I,J \in X$ with $I \not\subseteq J$, then there is $K \in X$ with $J \subseteq K$ and $K + I = A$.
\end{enumerate}
\end{lemma}

\begin{proof}
(\ref{I = A: 1}) Set $I =\ar{b^+, a^-}$. Since $0 \le a^+ \le b^+$, we have $a^+ \in I$, so $a = a^+ - a^- \in I$. Also, as $0 \le b^- \le a^-$, we have $b^- \in I$, and so $b \in I$. Thus, $b - a \in I$, and since $b-a$ is a nonzero real number, it is a unit in $A$, and hence $I = A$.

(\ref{I = A: 2}) Since $I \not\subseteq J$ there is $a \in I$ with $a \notin J$. Because $J$ is archimedean, by Proposition~\ref{prop: archimedean hull}, there is $n\ge 1$ with $(n|a|-1)^+ \notin J$. Let $K =  J \vee \ar{(n|a|-1)^-}$. Then $J \subseteq K$, and $I \vee K = A$ by (\ref{I = A: 1}). We show that $K \in X$. Otherwise $1 = x + y$ with $0 \le x, y$, $x \in J$, and $y \in \ar{(n|a|-1)^-}$. We claim that $y (n|a| - 1)^+ = 0$. To see this, we set $b = n|a|-1$. Since $y \in \ar{b^-}$, we have $(y-1/p)^+ \in \el{b^-}$ for each $p \ge 1$ by Proposition~\ref{prop: archimedean hull}. Therefore, by Remark~\ref{rem: [S]}, for each $p$ there is $m$ with $(y-1/p)^+ \le mb^-$. Thus, $0 \le (y-1/p)^+ b^+ \le mb^- b^+ = 0$ by Remark~\ref{rem: l-ring properties}(\ref{l-ring: plus meet minus}), and so $(y - 1/p)^+ b^+ = 0$. Because $y-1/p \le (y-1/p)^+$, we have $(y-1/p)b^+ \le (y-1/p)^+ b^+ = 0$, so $yb^+ \le (1/p) b^+$, which yields $pyb^+ \le b^+$. Since this is true for all $p \ge 1$, it follows that $yb^+ = y(|n|a-1)^+ = 0$ as $A$ is archimedean. This verifies the claim.
Therefore, $(n|a| - 1)^+  = (n|a| - 1)^+ (x + y) = (n|a| - 1)^+ x$, and so $(n|a| - 1)^+ 
\in J$, which is a contradiction. Thus, $K \in X$.
\end{proof}

\begin{lemma} \label{lem: f below g}
Let $A \in \bal$, $B$ be the free boolean extension of $\arch(A)$,  $X = \arch(A) \setminus \{A\}$, $I \in X$, 
$0 \le f,g \in D(\mathbb{R}[B])$, and $0 \le t \in \mathbb{R}$. Suppose that whenever $tx_{\lnot I} \le f$, there is $K \in X$ with $I \subseteq K$ and $tx_{\lnot K} \le g$. Then $f \le g$.
\end{lemma}

\begin{proof}
To show $f \le g$, by Lemma~\ref{lem: joins in D}(\ref{joins in D: 2}) we need to show that $tx_{\lnot I} \le f$ implies $tx_{\lnot I} \le g$. Given $t x_{\lnot I} \le f$, there is $K \supseteq I$ with $tx_{\lnot K} \le g$. If $K = I$, then we are done. Suppose $I \subset K$. For each $J \supseteq I$ with $J \vee K = A$ we have $tx_{\lnot J} \le tx_{\lnot I} \le f$, so there is $J' \supseteq J$ with $tx_{\lnot J'} \le g$. We have $J' \vee K = A$ since $J' \supseteq J$. We claim that
\[
I = K \cap \bigcap \{ J' \in X \mid J' \supseteq J, J' \vee K = A \}. 
\]
The inclusion $I \subseteq K \cap \bigcap \{J' \in X \mid J' \supseteq J, J' \vee K = A \}$ is clear since $I \subseteq J \subseteq J'$. For the reverse inclusion, let $a \in K \setminus I$. Then $|a| \in K \setminus I$, so we assume $0 \le a$. Since $I$ is archimedean, there is $n \ge 1$ with $(a - 1/n)^+ \notin I$. We show that $I \vee \ar{(a-1/n)^-} \ne A$. For, if $I \vee \ar{(a-1/n)^-} = A$, then $I + \el{(a-1/n)^-} = A$ by Lemma~\ref{lem: arch}. Therefore, by Lemma~\ref{lem: arch}(\ref{arch: 1}), there are $0 \le x,y$ with $x \in I$, $y \in \el{(a-1/n)^-}$, and $x + y = 1$. Thus, by Remark~\ref{rem: [S]}, we have $y \le m(a-1/n)^-$ for some $m \ge 1$, and hence $y(a-1/n)^+ = 0$ by Remark~\ref{rem: l-ring properties}(\ref{l-ring: plus meet minus}). Consequently, $(a-1/n)^+ = (a-1/n)^+x \in I$, a contradiction. Set $J = I \vee \ar{(a-1/n)^-}$. Because $\ar{a,(a-1/n)^-} = A$ by Lemma~\ref{lem: I = A}(\ref{I = A: 1}), we have $J \vee K = A$ since $a \in K$, and $a \notin J'$ because $J'$ is proper. Therefore, $a$ is not in the intersection. Thus, $I = K \cap \bigcap \{ J' \in X \mid J' \supseteq J, J' \vee K = A \}$ as desired. From this we obtain that in $B$ we have $\lnot I = \lnot K \vee \bigvee \{ \lnot J' \mid J' \supseteq J, J' \vee K = A \}$, and so
\[
tx_{\lnot I} = tx_{\lnot K} \vee \bigvee \{ tx_{\lnot J'} \mid J' \supseteq J, J' \vee K = A \} \le g.
\]
\end{proof}

We arrive at our final auxiliary lemma, item (\ref{not A: 2}) of which has the most involved proof.

\begin{lemma} \label{lem: not A}
Let $A \in \bal$, $X = \arch(A) \setminus \{A\}$, $0 \le a \in A$, and $I \in X$.
\begin{enumerate}
\item \label{not A: 1} If $s_I = \sup \{ r \mid (a-r)^- \in I \}$, then $(a - s_I)^- \in I$.
\item \label{not A: 2} $rx_{\lnot I} \le \alpha(a)$ iff $(a-r)^- \in I$.
\item \label{not A: 3} $s_I = \sup \{ r \mid rx_{\lnot I} \le \alpha(a) \}$ and $I \vee \ar{(a - s_I)^+} \ne A$. 
\end{enumerate}
\end{lemma}

\begin{proof}
(\ref{not A: 1}) If $(a-r)^- \in I$, then $a + I \ge r + I$ in $A/I$ by Remark~\ref{rem: positive part}(\ref{positive part: 2}). We use this to show that $(a - s_I)^- \in I$. For each $n \ge 1$ there is $r$ with $(a-r)^- \in I$ and $s_I - 1/n \le r$. Therefore, $(s_I - 1/n) + I \le r + I \le a + I$, and so $(s_I - a) + I \le 1/n + I$. Since this is true for all $n$, we have $(s_I-a) + I \le 0 + I$ as $A/I$ is archimedean. Thus, $s_I + I\le a + I$. Applying Remark~\ref{rem: positive part}(\ref{positive part: 2}) again yields $(a-s_I)^- \in I$. 

(\ref{not A: 2}) If $(a-r)^- \in I$, then $rx_{\lnot I} \le \alpha(a)$ by Lemma~\ref{lem: preservation of adding a scalar}(\ref{preservation of adding a scalar: 1}). Conversely, suppose that $rx_{\lnot I} \le \alpha(a)$. The result is clear if $r \le 0$ since then $(a-r)^- = 0 \in I$, so assume $r > 0$. By (\ref{not A: 1}) and Lemma~\ref{lem: preservation of adding a scalar}(\ref{preservation of adding a scalar: 1}), we may write $\alpha(a) = \bigvee \{ s_J x_{\lnot J} \mid J \in \arch(A)\}$, where $s_J$ is given in (\ref{not A: 1}). To show $(a-r)^- \in I$ we then need to show $r \le s_I$. We have $rx_{\lnot I} \le \bigvee \{ s_J x_{\lnot J} \mid J \in \arch(A) \}$ and $rx_{\lnot I} \le r$. Therefore,
\begin{align*}
rx_{\lnot I} &\le \bigvee \{ s_Jx_{\lnot J} \mid J \in \arch(A)\} \wedge r \\
&= \bigvee \{ s_Jx_{\lnot J} \wedge r \mid J \in \arch(A)\} \\
&= \bigvee \{ \min(s_J, r)x_{\lnot J} \mid J \in \arch(A) \}
\end{align*}
by Remark~\ref{rem: l-ring properties}(\ref{l-ring: vee wedge a}) and Lemma~\ref{rem: re meet sf}(\ref{re meet sf: 1}). To simplify notation set $t_J = \min(s_J, r)$. Then $(a - t_J)^- \le (a - s_J)^-$, so $(a - t_J)^- \in J$. From this and Remark~\ref{rem: Specker facts}(\ref{Specker facts: 3}) we get
\[
r = rx_I \vee rx_{\lnot I} \le rx_I \vee \bigvee \{ t_Jx_{\lnot J} \mid J \in \arch(A)\} \le r
\]
since the join is bounded by $r$, so equality holds throughout. 
Fix $\varepsilon > 0$. Then
\begin{align*}
r &= \left( rx_I \vee \bigvee \{ t_Jx_{\lnot J} \mid J \in \arch(A)\}\right) \vee (r - \varepsilon) \\
&= rx_I \vee \bigvee \{ t_Jx_{\lnot J} \vee (r - \varepsilon) \mid J \in \arch(A)\}.
\end{align*}
If $t_J \le r - \varepsilon$, then $t_J x_{\lnot J} \le r - \varepsilon$. Therefore, 
\begin{align*}
r &= rx_I \vee \bigvee \{ t_Jx_{\lnot J} \vee (r - \varepsilon) \mid J \in \arch(A)\} \\
&= \left( rx_I \vee \bigvee \{t_Jx_{\lnot J} \mid r - \varepsilon < t_J \}\right) \vee (r-\varepsilon).
\end{align*}
Thus, $rx_I \vee \bigvee \{ t_Jx_{\lnot J} \mid r - \varepsilon < t_J \} = r$ by Lemma~\ref{rem: re meet sf}(\ref{re meet sf: 3}). Multiplying both sides by $r^{-1}$ yields $x_I \vee \bigvee \{ r^{-1}t_Jx_{\lnot J} \mid r - \varepsilon < t_J \} = 1$. Consequently, $x_{\lnot I} \le \bigvee \{ r^{-1}t_Jx_{\lnot J} \mid r - \varepsilon < t_J\}$ by Remark~\ref{rem: idempotent facts}(\ref{idempotent facts: 2}), and hence $rx_{\lnot I} \le \{ t_Jx_{\lnot J} \mid r - \varepsilon < t_J\}$. Let $S = \{ J \in \arch(A) \mid r - \varepsilon < t_J \}$. We then have $rx_{\lnot I} \le \bigvee \{ t_J x_{\lnot J} \mid J \in S\}$ and $(a - (r-\varepsilon))^- \in J$ for each $J \in S$ since $(a - (r - \varepsilon))^- \le (a - t_J)^-$ and $(a - t_J)^- \in J$. Because $t_J \le r$ for each $J$ by definition and $r > 0$, we have $rx_{\lnot I} \le \bigvee \{ rx_{\lnot J} \mid J \in S\}$, so $x_{\lnot I} \le \bigvee \{ x_{\lnot J} \mid J \in S\}$. Therefore, $\lnot I \le \bigvee \{ \lnot J \mid J \in S\}$ in $B$. Since $B$ is a boolean algebra, $\bigwedge_B S \le I$. Because $\arch(A)$ is a sublattice of $B$, we have $\bigwedge_{\arch(A)} S \le \bigwedge_B S$. But $\bigwedge_{\arch(A)} S = \bigcap S$, so $\bigcap S \subseteq I$. As $(a-(r-\varepsilon))^- \in J$ for each $J \in S$, we see that $(a-(r-\varepsilon))^- \in I$. Since this is true for all $\varepsilon$, we have $a + I \ge (r - \varepsilon) + I$ for each $\varepsilon$,
so $a + I \ge r + I$ because $I$ is archimedean. Thus, $(a-r)^- \in I$.

(\ref{not A: 3}) We write $s = s_I$ for convenience. The first part of the statement follows from (\ref{not A: 1}) and (\ref{not A: 2}). Suppose that $I \vee \ar{(a-s)^+} = A$. Then $I + \el{(a-s)^+} = A$ by Lemma~\ref{lem: arch}. By Lemma~\ref{lem: arch}(\ref{arch: 1}) and Remark~\ref{rem: [S]}, there are $0 \le x,y$ with $x \in I$, $y \le n(a-s)^+$ for some $n$, and $x + y =1$. Then $1/n - y/n = x/n \in I$ and $1/n - y/n \ge 1/n - (a-s)^+$. Therefore, $1/n - y/n \ge (1/n - (a-s)^+) \vee 0$, so $(1/n - (a-s)^+)^+ \in I$. Using items (\ref{l-ring: negation and join}), (\ref{l-ring: vee + a}), (\ref{l-ring: vee wedge a}), and (\ref{l-ring: -a plus}) of Remark~\ref{rem: l-ring properties}, we have
\begin{align*}
\left( 1/n - (a-s)^+\right)^+ &= \left( 1/n - ((a-s)\vee 0)\right)^+ = \left( 1/n + ((s-a) \wedge 0)\right)^+ \\
&= \left((s + 1/n - a) \wedge 1/n\right) \vee 0 \\
&= \left((s + 1/n - a) \vee 0\right) \wedge 1/n \\
&= (a - (s + 1/n))^- \wedge 1/n.
\end{align*}
Thus, $(a - (s + 1/n))^- \wedge 1/n \in I$. Let $m \ge 1$ be such that $(a - (s + 1/n))^- \le m$. Applying Remark~\ref{rem: l-ring properties}(\ref{l-ring: join times scalar}) yields
\begin{align*}
(a - (s + 1/n))^- &\le mn(a - (s + 1/n))^- \wedge m  \\
&= mn[(a - (s + 1/n))^- \wedge 1/n] \in I.
\end{align*}
Therefore, $(a - (s + 1/n))^- \in I$, so $(s + 1/n)x_{\lnot I} \le \alpha(a)$ by Lemma~\ref{lem: preservation of adding a scalar}(\ref{properties of arch: 1}). This is a contradiction to the definition of $s = s_I$. Thus, $I \vee \ar{(a-s_I)^+} \ne A$.
\end{proof}

\def\cprime{$'$}
\providecommand{\bysame}{\leavevmode\hbox to3em{\hrulefill}\thinspace}
\providecommand{\MR}{\relax\ifhmode\unskip\space\fi MR }
\providecommand{\MRhref}[2]{%
  \href{http://www.ams.org/mathscinet-getitem?mr=#1}{#2}
}
\providecommand{\href}[2]{#2}

\end{document}